\documentclass[a4paper,11pt]{article}
\usepackage{pdfsync}

\usepackage[plainpages=false]{hyperref}
\usepackage{amsfonts,amsmath,amssymb,amsthm,enumerate}
\usepackage{latexsym,lscape,rawfonts,mathrsfs,mathtools}
\usepackage{enumitem,pifont}
\usepackage{cases}
\usepackage[savepos]{zref}

\usepackage[all]{xy}
\usepackage{eufrak}
\usepackage{makeidx}
\usepackage{graphicx,psfrag}
\usepackage{pstool}
\usepackage{tikz-cd}
\usepackage{smartdiagram}
\usepackage{caption}
\usepackage{array,tabularx,tabu}

\usepackage{setspace}

\usepackage[titletoc,title]{appendix}

\usepackage{txfonts}

\newcommand{\R}{\ensuremath{\mathbb{R}}}

\newcommand{\T}{\ensuremath{\mathbb{T}}}
\newcommand{\Z}{\ensuremath{\mathbb{Z}}}

\newcommand{\MM}{\mathcal{M}}

\newcommand{\ba}{\begin{align*}}
\newcommand{\ea}{\end{align*}}
\newcommand{\na}{\nabla}

\newcommand{\lc}{\left(}
\newcommand{\rc}{\right)}

\newcommand{\ep}{\epsilon}

\newcommand{\Rm}{\ensuremath{\mathrm{Rm}}}
\newcommand{\Rc}{\ensuremath{\mathrm{Rc}}}

\newcommand{\mmu}{\ensuremath{\boldsymbol{\mu}}}

\newcommand{\so}{\ensuremath{\mathrm{SO}(n)}}
\newcommand{\spin}{\ensuremath{\mathrm{Spin}(n)}}

\makeatletter
\newcommand*\owedge{\mathpalette\@owedge\relax}
\newcommand*\@owedge[1]{%
\mathbin{%
\ooalign{%
$#1\m@th\bigcirc$\cr
\hidewidth$#1\m@th\wedge$\hidewidth\cr
}%
}%
}
\makeatother

\makeatletter
\def\ExtendSymbol#1#2#3#4#5{\ext@arrow 0099{\arrowfill@#1#2#3}{#4}{#5}}

\makeatother

\makeatletter
\def\ExtendSymbol#1#2#3#4#5{\ext@arrow 0099{\arrowfill@#1#2#3}{#4}{#5}}
\newcommand\longright[2][]{\ExtendSymbol{-}{-}{\rightarrow}{#1}{#2}}
\makeatother

\def\XXint#1#2#3{{\setbox0=\hbox{$#1{#2#3}{\int}$ }
\vcenter{\hbox{$#2#3$ }}\kern-.55\wd0}}

\numberwithin{equation}{section}

\newtheorem{thm}{Theorem}[section]
\newtheorem{cor}[thm]{Corollary}
\newtheorem{prop}[thm]{Proposition}
\newtheorem{lem}[thm]{Lemma}

\newtheorem{rem}[thm]{Remark}

\newtheorem{defn}[thm]{Definition}

\title{Non-collapsing of Ricci shrinkers with bounded curvature}
\author{Conghan Dong \quad and \quad Yu Li}
\date{\today}

\begin{document}
\maketitle

\begin{abstract}
We establish a uniform entropy bound for simply connected Ricci shrinkers with a finite second homotopy group and a uniform curvature bound. Additionally, we extend the non-collapsing result to a broader class of smooth metric measure spaces satisfying Bakry-\'Emery conditions.
\end{abstract}

\section{Introduction}
A Ricci shrinker $(M^n, g, f)$ is a complete $n$-dimensional Riemannian manifold $(M^n,g)$ equipped with a smooth potential function $f: M \to \mathbb R$, satisfying the equation
\begin{equation} \label{E100}
\Rc+\na^2 f=\frac{1}{2}g,
\end{equation}
where $f$ is normalized so that
\begin{align} \label{E101}
R+|\nabla f|^2&=f.
\end{align}

As the critical metrics of Perelman's $\mmu$-functional, Ricci shrinkers play a central role in the study of singularity formation in the Ricci flow. In dimensions 2 and 3, all Ricci shrinkers have been fully classified; see \cite{Ha95} \cite{Pe02} \cite{Naber} \cite{NW} \cite{CCZ}, among others. The complete list consists of $\R^2,S^2,\R^3,S^3,S^2 \times \R$, and their quotients. Recently, all K\"ahler Ricci shrinkers on complex surfaces have also been classified (cf. \cite{CDS24} \cite{CCD24} \cite{BCCD24} and \cite{LW23}). 

To understand the uniform behavior of Ricci shrinkers, we consider their moduli space $\MM_n$, which is the space of all $n$-dimensional Ricci shrinkers equipped with the pointed-Gromov-Hausdorff distance. Investigating the compactification of $\MM_n$ is a problem of fundamental importance. In \cite{LLW21} and \cite{HLW21} (see also \cite{HM15} for the case $n=4$), it was shown that, under a uniform lower bound on entropy, a sequence of Ricci shrinkers converges to a singular Ricci shrinker with mild singularities. This demonstrates a weak compactness of $\MM_n$ under a uniform entropy bound.

The entropy $\mmu$ of a Ricci shrinker, see Definition \ref{dfn:201}, coincides with Perelman's celebrated entropy functional $\mmu(g,1)$ (see \cite{CN09} and \cite[Proposition $3$]{LW20}). The entropy encodes crucial information about Ricci shrinkers. For instance, it is shown in \cite[Theorem 1]{LW20} that $\mmu$ is closely related to the optimal Sobolev constant and the non-collapsing property of Ricci shrinkers. Moreover, $\mmu$ is always nonpositive and has a definite gap from zero for nontrivial Ricci shrinkers (see \cite{LW20} \cite{Yo09} and \cite{Yo12}). Additionally, $\mmu$ relates to the volume of the unit ball around the base point (see Lemma \ref{L202}). In essence, a uniform lower bound on entropy is equivalent to a uniform non-collapsing condition.

This paper focuses on another aspect: investigating the behavior of a sequence of Ricci shrinkers in $\MM_n$ with entropy tending to $-\infty$. Such sequences are collapsing, implying that the limit metric space must have a lower dimension. Understanding the properties of these limit spaces is critical for two main reasons: (1) the limit space reveals information about the boundary of $\MM_n$, crucial for deducing global properties of the moduli space (see, e.g., \cite{Hu20}); (2) analyzing the limit space can help rule out collapsing phenomena through contradiction arguments, provided the limit space is well-understood.

An example of this approach is the work of Petrunin, Rong, and Tuschmann \cite{PRT99} on a conjecture of Klingenberg and Sakai. In \cite{PRT99}, it was shown that a sequence of Riemannian manifolds with uniformly positively pinched sectional curvature and a stable topology (see \cite[Definition 0.2]{PRT99}) cannot collapse. Later, Fang and Rong \cite{FR99}, as well as Petrunin and Tuschmann \cite{PT99}, replaced the stability condition with a topological assumption: the manifolds are simply connected with finite second homotopy groups. The key idea in \cite{PRT99} was to construct a noncompact Alexandrov space with curvature bounded below by a positive constant using collapsing theory and a gluing argument. This led to a contradiction, as such Alexandrov spaces must be compact. Roughly speaking, this means that some form of positive curvature would prevent collapsing from occurring. Our main theorems (Theorem \ref{thm:001} and \ref{thm:002}) will share a similar spirit with respect to the Bakry-\'Emery Ricci curvature. See also \cite{Don23} for results related to positive scalar curvature.

Our main result is a non-collapsing theorem for Ricci shrinkers, assuming a uniform curvature bound and certain topological restrictions.

\begin{thm} \label{thm:001}
Let $(M^n, g, f)$ be a Ricci shrinker with $|\Rm| \le A$, where $M$ is simply connected and has a finite second homotopy group. Then, there exists a constant $C=C(n, A)$ such that 
\begin{align*} 
\mmu(g) \ge -C.
\end{align*}
\end{thm} 

Note that a similar result does not hold if the Ricci shrinker is replaced by an ancient solution to the Ricci flow with a type-I curvature bound and the same topological assumptions; see \cite{BKN12}.

The proof of Theorem \ref{thm:001} relies on a contradiction argument. Suppose there exists a sequence of Ricci shrinkers $(M_i^n, g_i, f_i,p_i)$ satisfying the given assumptions with $\mmu(g_i) \to -\infty$. Our goal is to derive a contradiction.

By taking a subsequence, we assume that $(M_i^n, g_i, f_i,p_i)$ converges in the pointed Gromov-Hausdorff sense to a limit metric space $(X, d_X, f_X, p_X)$. The first task is to analyze the structure of the limit space $X$. Since the curvature is uniformly bounded, we apply the collapsing theory with bounded sectional curvature, initially developed through Gromov's celebrated theorem on almost flat manifolds \cite{Gro78} and further advanced by Cheeger, Fukaya, and Gromov in the 1980s and 1990s (see \cite{CG86} \cite{CG90} \cite{Fu87} \cite{Fu88} \cite{Fu89} and \cite{CFG92}, among others). It can be shown (see Proposition \ref{prop:po2}) that $(X, d_X)$ is a smooth Riemannian orbifold, except for a singular set of high codimension.

Furthermore, for sufficiently large $i$,  each $M_i$ locally admits a pure $N$-structure arising from a singular fibration. Specifically, we consider the oriented frame bundles $FM_i$ equipped with the bundle metrics $\tilde g_i$, the lifted function $\tilde f_i$ and the lifted base points $\tilde{p}_i$. The sequence $(FM_i, \tilde g_i, \so, \tilde f_i, \tilde p_i)$ converges, in the equivariant pointed Gromov-Hausdorff sense, to a limit space $(Y,g_Y, \so, f_Y, p_Y)$, where $(Y, g_Y)$ is a smooth Riemannian manifold, and the quotient $(Y, g_Y)/\so$ is isometric to $(X, d_X)$.

For a fixed large number $s>0$, there exists an $\so$-equivariant fibration $\tilde \Phi_i: \tilde U_i(s) \to \tilde \Omega(s)$, where $\tilde \Omega(s)=\{x \in Y \mid f_Y(x) \le s\}$, and $\tilde U_i (s)$ closely approximates $\tilde \Omega_i(s)=\{x \in FM_i \mid \tilde f_i(x) \le s\}$; see Theorem \ref{thm:fib}. The singular fibration on $U_i(s) \subset M_i$, where $U_i(s)=\tilde U_i(s) /\so$, is given by the quotient of $\tilde \Phi_i$. In particular, there exists a singular fibration $\Phi_i: U_i(s) \to \Omega_i(s)$, where $\Omega_i(s)=\{x \in X \mid f_X(x) \le s\}$, such that the following diagram commutes:
\[\begin{tikzcd} 
{(\tilde U_i(s),\tilde{g}_i)} && {(\tilde \Omega(s),d_{g_Y})} \\
\\
{(U_i(s),g_i)} && {(\Omega(s),d_X)}
\arrow["\Phi_i", from=3-1, to=3-3]
\arrow["{\mathrm{mod}\ \so}", from=1-3, to=3-3]
\arrow["{\mathrm{mod}\ \so}", from=1-1, to=3-1]
\arrow["\tilde \Phi_i", from=1-1, to=1-3]
\end{tikzcd}\]

For sufficiently large $s$, $U_i(s)$ captures the topological structure of $M_i$. Specifically, as shown in Proposition \ref{prop:topo3}, $U_i(s)$ is homotopy equivalent to $M_i$, and hence is simply connected by our assumption. A result of Rong \cite{Rong96} (see Propositions \ref{prop:topo4} and \ref{prop:topo5}) implies that the singular fibration $\Phi_i$ is realized by a $T^k$-action on $U_i(s)$, where the orbits correspond to the singular fibers of $\Phi_i$. Moreover, one can construct a $T^k$-invariant metric $g_i'$ and a function $f_i'$ that are close to the original metric $g_i$ and $f_i$; see Proposition \ref{prop:ave}.

For simplicity, we assume $U_i(s)=M_i$, $\tilde U_i(s)=FM_i$, $\Omega(s)=X$, and $\tilde \Omega(s)=Y$, for a sufficiently large fixed $s$. The general case involves choosing an appropriate sequence of $s_j \to \infty$ and taking a diagonal subsequence; see Theorem \ref{thm:dia}. 

The assumption on the second homotopy group of $M_i$ implies that the collapsing $M_i \to X$ is stable, in the sense that all corresponding $T^k$-actions are equivalent up to  conjugations of $\mathrm{GL}(\Z, k)$ on $T^k$. Consequently, all $M_i$ are diffeomorphic to one another. We may therefore write $M=M_i$, with $g'_i$ and $f_i'$ as metrics and functions on $M$, respectively.

To derive a contradiction, we follow the approach in \cite{PRT99} by constructing a metric space $W$, which is an $\R$-bundle over $X$. This involves unwrapping a fixed circle component within each singular fiber, which is a quotient of $T^k$. This unwrapping process is the most technical part of the paper, and we provide full details for its implementation. Briefly, we consider the principal circle bundle $F/T^{k-1} \to Y$, where $F$ denotes either the oriented frame bundle over $M$ or the $\spin$-bundle over $M$. We then unwrap the circle fibers to form an $\R$-bundle $Z$ over $Y$. In general, this unwrapping process is obstructed by the Euler class of the circle bundle. However, as observed in \cite{PRT99}, the Euler class vanishes in the limit, as the length of the circle fibers tends to zero. This allows us to construct an $\R$-bundle $W$ over $X$, obtained as the quotient of $Z$ by $\so$ or $\spin$. 

By construction, $W$ with the induced metric is a noncompact, complete metric space; see Lemma \ref{L404}. At this point, the Ricci shrinker equation \eqref{E100} becomes crucial. On the regular part $(\mathcal R_W, g_W)$ of $W$, we derive the equation:
\begin{align} \label{eq:main}
\Rc(g_W)+\na_{g_W}^2(\bar f_W) \ge \frac{1}{2} g_W.
\end{align}
where $\bar f_W$ is a smooth, $\R$-invariant function on $\mathcal R_W$, obtained as the limit of $f'_i$ plus a density function. The desired contradiction then follows from the second variational formula for minimizing geodesics connecting two regular points on an $\R$-orbit in $W$.

As a direct corollary to Theorem \ref{thm:001}, we have:

\begin{cor} \label{cor:001}
Given $n$ and $A$, there are only finitely many diffeomorphism types of simply connected $n$-dimensional manifolds $M$ with finite second homotopy groups that admit a Ricci shrinker metric satisfying $|\Rm| \le A$.
\end{cor}

The method used to prove Theorem \ref{thm:001} can be extended to a broader class of smooth metric measure spaces. To define this class, we introduce the following:

\begin{defn} \label{def:main}
Given a constant $\kappa>0$, let $\mathcal N(n, A, \kappa)$ be a class of pointed smooth metric measure spaces $(M^n, g, f, p)$ satisfying 
\begin{enumerate}[label=\textnormal{(\roman{*})}]
\item $(M^n, g)$ is a complete $n$-dimensional smooth Riemannian manifold with $|\Rm| \le A$.

\item $f$ is a $C^2$-function on $M$ such that $p$ is minimum point, $f(p)=0$, and $|\na^2 f| \le A$.

\item The Bakry-\'Emery condition holds: $\Rc+\na^2 f \ge \kappa g$.

\item $M$ is simply connected and has a finite second homotopy group.
\end{enumerate}
\end{defn}

\begin{thm} \label{thm:002}
There exists a constant $v=v(n, A,\kappa)>0$ such that for any $(M^n, g, f, p) \in \mathcal N(n, A, \kappa)$,
\begin{align*} 
\mathrm{Vol}_g \lc B_g(p, 1) \rc \ge v.
\end{align*}
\end{thm} 

It follows directly that any Ricci shrinker in Theorem \ref{thm:001} belongs to $\mathcal N(n, C(n)A+1/2, 1/2)$ if the potential function $f$ is shifted by a constant to satisfy $f(p)=0$. Thus, Theorem \ref{thm:002} extends Theorem \ref{thm:001}. Furthermore, Theorem \ref{thm:002} also generalizes \cite[Theorem 0.4]{PT99}, in which the potential function $f$ is identically zero.

There are two main differences between Ricci shrinkers and metric measure spaces defined above. First, the spaces in $\mathcal N(n, A, \kappa)$ satisfy only the Bakry-\'Emery condition as in Definition \ref{def:main} (iii), which is weaker than the equality in the Ricci shrinker equation \eqref{E100}. However, this is not a significant issue because the lower bound $\kappa g$ is sufficient for deriving the key inequality \eqref{eq:main}. The second difference is that spaces in $\mathcal N(n, A, \kappa)$ do not inherently satisfy higher-order estimates for $\Rm$ or $f$. In the Ricci shrinker case, such estimates are guaranteed by Shi's estimates and the Ricci shrinker equation (see \eqref{eq:curbound1} and \eqref{eq:fbound}).

To address these challenges, we smooth the metric measure space $(M^n, g, f, p) \in \mathcal N(n, A, \kappa)$ using Ricci flow, where $g$ serves as the initial metric, and smooth $f$ by solving the heat equation under the Ricci flow. Through careful estimates, we construct a new metric $g'$ and function $f'$ on $M$ such that $(M^n, g', f', p') \in \mathcal N(n, A', \kappa/2)$, where $A'=A'(n, A, \kappa)$ is a constant, and $g'$ and $f'$ satisfy uniform higher-order estimates. Moreover, the volume of the unit ball around $p$ with respect to $g$ is uniformly comparable to the volume of the unit ball around $p'$ with respect to $g'$; see Theorem \ref{thm:smooth}. This allows us to apply the same contradiction argument used in the proof of Theorem \ref{thm:001} to establish Theorem \ref{thm:002}.

The structure of this paper is as follows. In Section 2, we review fundamental concepts and properties of Ricci shrinkers. Section 3 addresses the general behavior of collapsing sequences of Ricci shrinkers with bounded curvature, using the collapsing theory for spaces with bounded sectional curvature. In Section 4, we prove the main Theorem \ref{thm:001} through the gluing process described earlier. Finally, Section 5 extends these methods to prove Theorem \ref{thm:002}.
\\
\\
\\
{\bf Acknowledgements}: 
We would like to thank Ruobing Zhang for helpful discussions on the collapsing theory. Conghan Dong is supported by the NSF grant DMS-1928930, while he was in residence at the Simons Laufer Mathematical Sciences Institute (formerly MSRI) in Berkeley, California, during the Fall 2024 semester. Yu Li is supported by YSBR-001, NSFC-12201597, and research funds from University of Science and Technology of China and Chinese Academy of Sciences. 

\section{Preliminaries}

In this section, we review some basic concepts and fundamental properties of Ricci shrinkers.

Let $(M^n, g, f)$ be a Ricci shrinker with normalization \eqref{E101}. From \cite[Corollary $2.5$]{CBL07}, it follows that the scalar curvature $R \ge 0$, with $R>0$ unless $(M^n,g)$ is isometric to the Gaussian soliton $(\R^n,g_E)$, as guaranteed by the strong maximum principle.

For the potential function $f$, we have the following fundamental estimates.

\begin{lem}[\cite{CZ10} \cite{HM11}]
\label{L201}
Let $(M^n,g,f)$ be a Ricci shrinker. Then $f$ satisfies the quadratic growth estimate:
\begin{align*}
\frac{1}{4}\left(d(x,p)-5n \right)^2_+ \le f(x) \le \frac{1}{4} \left(d(x,p)+\sqrt{2n} \right)^2
\end{align*}
for all $x\in M$, where $p$ is a minimum point of $f$ and $a_+ :=\max\{0,a\}$.
\end{lem}

It follows from Lemma \ref{L201} that any minimum point lies in a ball of fixed size. In particular, we can designate a minimum point $p$ as the base point of the Ricci shrinker and represent it as $(M^n, g, f, p)$.

Next, we recall the definition of the entropy of a Ricci shrinker.

\begin{defn}\label{dfn:201}
Let $(M^n,g,f)$ be a Ricci shrinker with normalization \eqref{E101}. Its entropy is defined as
\begin{align*}
\mmu=\mmu(g)\coloneqq \log \int_M \frac{e^{-f}}{(4\pi)^{n/2}}\, dV_g.
\end{align*}
\end{defn}

The entropy $\mmu$ coincides with Perelman's celebrated entropy functional $\mmu(g,1)$; see \cite{CN09} and \cite[Proposition $3$]{LW20}. Furthermore, $\mmu$ is always nonpositive and finite for any Ricci shrinker; see \cite[Theorem $1$]{LW20}. For a detailed discussion of the $\mmu$-functional on Ricci shrinkers and its properties, we refer readers to \cite[Section $5$]{LW20}.

The following characterization of $\mmu$, proved in \cite[Lemma $2.5$]{LLW21}, states that the unit ball around the base point collapses if and only if the entropy approaches $-\infty$.
\begin{lem}
\label{L202}
Let $(M^n,g,f,p)$ be a Ricci shrinker. Then, there exists a constant $C=C(n)>1$ such that
\begin{align*}
C^{-1} e^{\mmu} \le \mathrm{Vol}_g \lc B_g(p, 1) \rc \le C e^{\mmu}
\end{align*}
\end{lem}

In this paper, we focus on Ricci shrinkers with uniformly bounded curvature. This naturally leads to the consideration of the following moduli space:

\begin{defn}\label{defn-E}
Let $\mathcal E(n,A)$ denote the family of Ricci shrinkers $(M^n,g,f)$ satisfying the curvature bound
\begin{align*}
|\Rm| \le A.
\end{align*}
\end{defn}

We collect some elementary properties for Ricci shrinkers in this moduli space.

\begin{lem}
\label{L203}
Let $(M^n,g,f,p) \in \mathcal E(n,A)$ be a Ricci shrinker, and define $\Omega_M(s):=\{x \in M \mid f(x) \le s\}$. The following conclusions hold:
\begin{enumerate}[label=\textnormal{(\alph{*})}]
\item There exists a constant $L=L(n, A) \ge 1$ such that for any $s_2>s_1 \ge L$, the set $\Omega_M(s_2) \setminus \mathrm{int}(\Omega_M(s_1))$ is diffeomorphic to $\partial \Omega_M(s_2) \times [0,1]$. Moreover, $M$ is diffeomorphic to $\mathrm{int}(\Omega_M(s))$ for any $s \ge L$.

\item For any $l \ge 0$, there exists a constant $C_l=C_l(n, A)>0$ such that
\begin{align} \label{eq:curbound1}
|\na^l \Rm| \le C_l.
\end{align}

\item For any $l \ge 2$, there exists a constant $C_l=C_l(n, A)>0$ such that
\begin{align} \label{eq:fbound}
|\na^l f| \le C_l.
\end{align}
\end{enumerate}
\end{lem}

\begin{proof}
(a): By assumption, the scalar curvature satisfies $R \le C(n) A$. Thus, the conclusion follows directly from the identity \eqref{E101} and Lemma \ref{L201}.

(b): Recall that every Ricci shrinker is associated with a self-similar ancient solution to the Ricci flow. Specifically, let ${\psi^t}: M \to M$ be a family of diffeomorphisms generated by $(1-t)^{-1}\nabla f$, with $\psi^{0}=\text{id}$. The rescaled pull-back metric $g(t)\coloneqq (1-t) (\psi^t)^*g$ satisfies the Ricci flow equation for any $t \in (-\infty ,1)$. Therefore, the conclusion follows from Shi’s curvature estimates \cite{Shi89A}.

(c): This follows immediately from (b) and the Ricci shrinker equation \eqref{E100}.
\end{proof}

\begin{prop} \label{prop:smooth}
Let $\mathcal E_B(n,A)$ be the subspace of $\mathcal E(n,A)$  where the entropy is bounded below by $-B$. Then, $\mathcal E_B(n,A)$ is compact in the pointed $C^{\infty}$ topology in the sense of Cheeger-Gromov.
\end{prop}

\begin{proof}
Let $(M_i, g_i, f_i, p_i) \in \mathcal E_B(n,A)$. By Lemma \ref{L203} (b) and \cite[Theorem 1]{LLW21}, we can take a subsequence if necessary such that
\begin{align*}
	(M_i, g_i,f_i ,p_i) \longright{\text{pointed }{C}^{\infty}} (M_{\infty}, g_{\infty}, f_{\infty},p_{\infty}),
\end{align*}
where $(M_{\infty}, g_{\infty}, f_{\infty})$ is a Ricci shrinker.  Moreover, by smooth convergence and \cite[Proposition 8.8]{LLW21}, it follows that $(M_{\infty}, g_{\infty}, f_{\infty}) \in \mathcal E_B(n,A)$.
\end{proof}

In particular, combining Proposition \ref{prop:smooth} with Lemma \ref{L203} (a), we conclude that $\mathcal E_B(n,A)$ contains only finitely many diffeomorphism types.

\section{Collapsing of Ricci shrinkers with bounded curvature}

In this section, we discuss the behavior of a sequence of Ricci shrinkers collapsing under uniformly bounded sectional curvature.

To begin with, we first recall the following concept:

\begin{defn}[Frame bundle]
Let $(M^n, g)$ be an orientable Riemannian manifold. Its oriented orthogonal frame bundle $FM$ is equipped with a canonical metric $\tilde g$, defined as follows: the Levi-Civita connection of $g$ determines the horizontal distribution, while each fiber is endowed with the standard bi-invariant metric on $\mathrm{SO}(n)$. In particular, the projection $\rho: (FM, \tilde g) \to (M, g)$ is a Riemannian submersion.
\end{defn}

For an orientable Ricci shrinker $(M^n,g,f,p)$, we assign a base point $\tilde p$ in its frame bundle $FM$ such that $\rho(\tilde p)=p$. Furthermore, the function $f$ can be lifted to a smooth, $\mathrm{SO}(n)$-invariant function $\tilde f$ on $FM$. Naturally, $\tilde \Omega_M(s):=\{x \in FM \mid \tilde f(x) \le s\}$ corresponds to the frame bundle of $\Omega_M(s)=\{x \in M \mid f(x) \le s\}$.

The following lemma is a direct consequence of Lemma \ref{L203} (b) and the celebrated O'Neill formula \cite{One66}. Recall that $\mathcal{E}(n,A)$ is defined in Definition \ref{defn-E}.

\begin{lem}
\label{L204}
Let $(M^n,g,f,p) \in \mathcal E(n,A)$ be an orientable Ricci shrinker. For any $l \ge 0$, there exists a constant $C_l=C_l(n, A)>0$ such that on $(FM, \tilde g)$,
\begin{align*}
|\na_{\tilde g}^l \Rm(\tilde g)| \le C_l.
\end{align*}
\end{lem}

For the remainder of this section, we fix a sequence of orientable Ricci shrinkers $(M_i, g_i, f_i, p_i) \in \mathcal E(n, A)$ with entropy $\mmu(g_i) \to -\infty$. 

Using the standard precompactness theory \cite{BBI01} (see also \cite[Theorem 1.1]{LLW21}), we obtain the following result:

\begin{prop}
By taking a subsequence, we have
\begin{align} \label{eq:conv1}
	(M_i, g_i,f_i ,p_i) \longright{\text{pointed Gromov-Hausdorff}} (X, d_X, f_X,p_X),
\end{align}
where $(X, d_X)$ is a geodesic space, and $f_X$ is a locally Lipschitz function on $X$.
\end{prop}

By our assumption on the entropy, the Hausdorff dimension of $X$ is strictly less than $n$. Furthermore, by the curvature bound, $(X, d_X)$ is an Alexandrov space with curvature bounded below by $-A$.

The local structure around a point $x \in X$ can be described as follows; see \cite{Fu88}, \cite{Fu90} and \cite{Rong10}. Take a sequence $x_i \in M_i$ with $x_i \to x$. Fix a small ball $B(0,\ep) \subset \R^n \simeq T_{x_i} M_i$, where $\ep$ is a fixed constant smaller than $\pi/\sqrt A$. Consider the pull back metric $\exp_{x_i}^* g_i$, denoted simply by $g_i$ for convenience. Define the local fundamental group $G_i$ as
\begin{align*}
G_i=\{\gamma \mid \gamma \text{ is a geodesic loop at } x_i \text{ with length smaller than } \ep\}.
\end{align*}
Then, we have the equivariant convergence
\begin{align*}
	\lc B(0, \ep), g_i,G_i \rc \longright{\text{equivariant\ } C^{\infty} } \lc B(0, \ep) ,g, G_\infty \rc.
\end{align*}
The definition of equivariant convergence can be found in \cite[Definition 6.11]{Fu90} or \cite[Definition 1.6.6]{Rong10}. Here, $G_\infty$ is a local Lie group acting smoothly on $B(0, \ep)$, and its Lie algebra $\mathfrak g$ is nilpotent. 

The following isometry holds:
\begin{align*}
\lc B(x, \ep), d_X \rc=\lc B(0, \ep), g \rc/G_\infty. 
\end{align*}
The isotropy group at $0$, denoted by $I_x$, is a Lie group which is a finite extension of a torus. By the standard slice theorem, and by choosing a smaller $\ep$ if necessary, there exists a slice $S=\exp_0(B')$, where $B'$ is a small ball in $\lc T_0(G_\infty \cdot 0) \rc^{\perp}$, such that $B(x, \ep)$ is isometric to $S/I_x$.

We now consider the associated frame bundles of this sequence. The following result follows from Lemma \ref{L204} and \cite{Fu88}.

\begin{prop}\label{prop:frame}
By taking a subsequence, we have
\begin{align} \label{eq:conv2}
	(FM_i, \tilde g_i, \mathrm{SO}(n),\tilde f_i ,\tilde p_i) \longright{\text{equivariant pointed Gromov-Hausdorff}} (Y, g_Y,\mathrm{SO}(n), f_Y,p_Y),
\end{align}
where $(Y,g_Y)$ is a smooth Riemannian manifold. Furthermore, $(Y,g_Y)/\mathrm{SO}(n)$ is isometric to $(X,d_X)$.
\end{prop}

Let $\rho: Y \to Y/\mathrm{SO}(n)=X$ be the projection. For any $y \in Y$, the isotropy group of $\mathrm{SO}(n)$ at $y$ is isomorphic to $I_{\rho(y)}$; see \cite[Theorem 1.6]{Fu88}. Additionally, $f_Y$ is the pullback of $f_X$ under the projection $\rho $ and is therefore $\mathrm{SO}(n)$-invariant. In fact, we have the following stronger result.

\begin{prop} \label{prop:po1}
For the limit function $f_Y$ in \eqref{eq:conv2}, the following properties hold:
\begin{enumerate}[label=\textnormal{(\alph{*})}]
\item For any $y \in Y$,
\begin{align} \label{eq:id00}
\frac{1}{4} \lc d_{g_Y}(y,p_Y)-C(n) \rc^2_+ \le f_Y(y) \le \frac{1}{4} \lc d_{g_Y}(y,p_Y)+C(n) \rc^2.
\end{align}

\item $f_Y$ is smooth and satisfies
\begin{align} \label{eq:id0}
0\le f_Y-|\na f_Y|^2 \le C(n)A.
\end{align}
\end{enumerate}
\end{prop}

\begin{proof}
(a): Since the diameter of each $\mathrm{SO}(n)$-fiber of $FM_i$ is uniformly bounded by $C(n)$, we have
\begin{align*}
d_{g_i} (p_i, x)-C(n) \le d_{\tilde g_i} (\tilde p_i, \tilde x) \le d_{g_i} (p_i, x)+C(n)
\end{align*}
for any $x \in M_i$ and any lift $\tilde x$ of $x$. Therefore, \eqref{eq:id00} follows from Lemma \ref{L201} by taking the limit.

(b): For any $y \in Y$, we take a sequence $y_i \in FM_i$ with $y_i \to y$. Fix a ball $B(0, \ep) \subset \R^{\frac{n(n+1)}{2}}=T_{y_i}(FM_i)$. As before, we have smooth convergence:
\begin{align*}
	\lc B(0, \ep), \tilde g_i,G_i \rc \longright{\text{equivariant } C^{\infty}} \lc B(0, \ep) ,g_{\infty}, G_{\infty} \rc,
\end{align*}
where $G_i$ is the local fundamental group. By \eqref{eq:fbound}, the lift of $\tilde f_i$ on $B(0,\ep)$, still denoted by $\tilde f_i$, converges smoothly to a function $f_{\infty}$. Since each $\tilde f_i$ is $G_i$-invariant, we conclude that $f_{\infty}$ is $G_{\infty}$-invariant. Therefore, $f_Y$ is smooth around $y$, as $f_Y$ on $B(y, \ep)$ is the quotient of $f_{\infty}$ on $B(0, \ep)/G_{\infty}$.

Additionally, from \eqref{E101} and the curvature bound, each $\tilde f_i$ satisfies
\begin{align*}
0 \le \tilde f_i-|\na_{\tilde g_i} \tilde f_i|^2 \le C(n) A.
\end{align*}
Taking the limit, we obtain
\begin{align*}
0 \le f_{\infty}-|\na_{g_{\infty}} f_{\infty}|^2 \le C(n) A.
\end{align*}
Since the action of $G_{\infty}$ on $B(0, \ep)$ is free and isometric, \eqref{eq:id0} follows.
\end{proof}

Note that the limit space $X$ can be identified as the orbit space of $(Y, \mathrm{SO}(n))$. Consequently, $X$ has a natural stratification by orbit types. In particular, there exists an open and dense set $\mathcal R$ of $X$, consisting of the projections of all principal orbits in $Y$, such that $\mathcal R$ admits a smooth manifold structure. Furthermore, the principal isotropy group is finite (see the proof of Proposition \ref{prop:topo4} (a) below).

The following strong characterization is given by \cite{NT08}.

\begin{prop} \label{prop:po2}
$X$ admits a smooth Riemannian orbifold structure away from a closed set $\mathcal S$ of Hausdorff dimension at most $\min\{n-5, \mathrm{dim}(X)-3\}$. Moreover, $f_X$ is smooth on $X \setminus\mathcal S$ in the orbifold sense.
\end{prop}

\begin{proof}
The first conclusion follows from \cite[Theorem 1.1]{NT08}. For any $x \in X \setminus\mathcal S$, let $y \in Y$ satisfy $\rho(y)=x$. Using the slice theorem, there exists a slice $S_y$ such that $S_y/H$, where $H$ is the isotropy group of $\mathrm{SO}(n)$ at $y$, represents a neighborhood of $x$. As shown in the proof of \cite[Theorem 1.1]{NT08}, there exists a local Riemannian reduction; see \cite[Definition 3.3]{NT08}. Specifically, there exists a submanifold $S_y' \subset S_y$ and a finite subgroup $\Gamma \le H$ acting on $S_y'$ such that $(S_y',\Gamma)$ forms a Riemannian orbifold chart around $x$. By Proposition \ref{prop:po1}, the restriction of $f_Y$ to $S_y'$ is smooth. Consequently, $f_X$ is smooth at $x$ in the orbifold sense.
\end{proof}

For later applications, we introduce the following auxiliary sets:

\begin{defn}
Under the above assumptions, we define:
\begin{align*}
\Omega_i(s):=\{x \in M_i \mid f_i(x) \le s\},& \quad \tilde \Omega_i(s):=\{x \in FM_i \mid \tilde f_i(x) \le s\}, \\
\Omega(s):=\{x \in X \mid f_X(x) \le s\}, &\quad \tilde \Omega(s):=\{x \in Y \mid f_Y(x) \le s\}.
\end{align*}
\end{defn}

We also need the following concept:

\begin{defn}\label{dfn:205}
Let $(M, G)$ and $(N, G)$ be two smooth manifolds with a Lie group $G$-action.
\begin{enumerate}[label=\textnormal{(\roman{*})}]
\item $(M, G)$ and $(N, G)$ are called \textbf{$G$-diffeomorphic} if there exists a $G$-equivariant diffeomorphism $\phi: M \to N$.

\item $(M, G)$ and $(N, G)$ are called \textbf{weakly $G$-diffeomorphic} if there exists a diffeomorphism $\phi: M \to N$ and an automorphism $a$ of $G$ such that for any $x \in M$ and $\xi \in G$, 
\begin{align*}
\phi(\xi x)=a(\xi ) \phi(x)
\end{align*}
\end{enumerate}
\end{defn}

\begin{lem} \label{lem:topo2}
There exists a constant $L=L(n, A)>0$ such that the following properties hold:
\begin{enumerate}[label=\textnormal{(\alph{*})}]
\item For any $s_2> s_1 \ge L$, the set $\tilde \Omega(s_2) \setminus \mathrm{int}( \tilde \Omega(s_1) )$ is $\mathrm{SO}(n)$-diffeomorphic to $\partial \tilde \Omega(s_1) \times [0,1]$, where the $\mathrm{SO}(n)$-action on $[0,1]$ is trivial. Moreover, $Y$ is $\mathrm{SO}(n)$-diffeomorphic to $\mathrm{int} (\tilde \Omega(s))$ for any $s \ge L$.

\item For any $s_2> s_1 \ge L$, the set $\Omega(s_2) \setminus \mathrm{int}(\Omega(s_1) )$ is homeomorphic to $\partial \Omega(s_1) \times [0,1]$. Moreover, $X$ is homeomorphic to $ \mathrm{int}(\Omega(s))$ for any $s \ge L$.
\end{enumerate}
\end{lem}

\begin{proof}
(a): By Proposition \ref{prop:po1}, $f_Y$ has no critical point when $f_Y \ge L(n, Y)$. Consequently, the conclusion follows from the standard gradient flow method applied to $f_Y$, noting that $f_Y$ is $\mathrm{SO}(n)$-invariant.

(b): The $\so$-diffeomorphisms established in part (a) naturally descend to the required homeomorphisms.
\end{proof}

The following fibration theorem is well-known; see \cite[Theorem 9.1]{Fu88}, \cite[Theorem 2.6]{CFG92}, \cite[Theorem 0-1]{Fu89} and \cite[Theorem 5.7.1]{Rong10}. We remark that $\tilde \Phi_i$ below depends on $s$.

\begin{thm}\label{thm:fib}
Given $s \ge L(n, A)$ and sufficiently large $i$, there exists a smooth $\mathrm{SO}(n)$-equivariant fibration $\tilde \Phi_i: \tilde U_i(s) \to \tilde \Omega(s)$, where $\tilde U_i (s)$ is an $\mathrm{SO}(n)$-invariant compact set satisfying $\tilde \Omega_i(s-1) \subset \tilde U_i(s) \subset \tilde \Omega_i(s+1)$, with the following properties:
\begin{enumerate}[label=\textnormal{(\alph{*})}]
\item $\tilde \Phi_i$ is an almost Riemannian submersion, meaning there exists a sequence $\ep_i \to 0$ such that for any horizontal vector $\zeta$, 
\begin{align*}
e^{-\ep_i} |\zeta| \le |\mathrm{d}\tilde \Phi_i(\zeta)| \le e^{\ep_i} |\zeta|.
\end{align*}

\item Each fiber of $\tilde \Phi_i$ is connected, with diameter bounded by $\ep_i$. Moreover, the second fundamental form of the fibers is uniformly bounded.

\item The fibers admit canonical flat connections that vary smoothly, and the $\mathrm{SO}(n)$-action preserves these flat connections. Furthermore, there exists a simply connected nilpotent group $N$ and a group of affine transformations $\Gamma$ of $N$ such that $\tilde\Phi_i^{-1}(y)$ is affine equivalent to $N/\Gamma$, with $[\Gamma: \Gamma\cap N]<\infty$. In particular, the structure group of $\tilde \Phi_i$ is contained in 
\begin{align*}
C(N)/\lc C(N) \cap \Gamma \rc \rtimes \mathrm{Aut}(\Gamma),
\end{align*}
where $C(N)$ is the center of $N$.
\end{enumerate}
\end{thm}

It is clear that $\tilde \Phi_i$ induces a singular fibration $\Phi_i: U_i(s) \to \Omega(s)$, where $U_i(s):=\rho_i(\tilde U_i(s))$ for $\rho_i: FM_i \to M_i$. The following commutative diagram illustrates this relationship:
\[\begin{tikzcd} 
{(\tilde U_i(s),\tilde{g}_i)} && {(\tilde \Omega(s),d_{g_Y})} \\
\\
{(U_i(s),g_i)} && {(\Omega(s),d_X)}
\arrow["\Phi_i", from=3-1, to=3-3]
\arrow["{\rho}", from=1-3, to=3-3]
\arrow["{\rho_i}", from=1-1, to=3-1]
\arrow["\tilde \Phi_i", from=1-1, to=1-3]
\end{tikzcd}\]

For the singular fibration, we have the following result, the proof of which can be found in \cite[Theorem 2.1]{NT08}.

\begin{lem} \label{lem:sinfi}
	For any $t<s$, there exists a constant $C=C(n, A,t, Y)>1$ independent of $i$ and $s$ such that for any $x \in \Omega(t)$, $\Phi_i^{-1}(x)$ is a smooth submanifold with the second fundamental form bounded above by $C$ and normal injectivity radius bounded below by $C^{-1}$.
\end{lem}

Next, we investigate the topology of $U_i(s)$ in relation to $M_i$.

\begin{prop} \label{prop:topo3}
There exists a constant $L=L(n, A)>0$ such that for any $s \ge L$, if $i$ is sufficiently large, $\partial U_i(s)$ is $h$-cobordant to $\partial \Omega_i(s)$. Moreover, $U_i(s)$ is homotopy equivalent to $M_i$.
\end{prop}

\begin{proof}
From Lemma \ref{lem:topo2} (a), for any $s/2 \le t< s$, the region $\tilde \Omega(s) \setminus \mathrm{int}(\tilde \Omega(t))$ is $\mathrm{SO}(n)$-diffeomorphic to $\partial \tilde \Omega(t) \times [0,1]$. Define $\tilde \Sigma_i(t):=\tilde \Phi_i^{-1}(\partial \tilde \Omega(t))$. By \cite[Theorem 3.1]{Bi73}, the fiber bundle $\tilde\Phi_i$ over $\tilde \Omega(s) \setminus \mathrm{int}(\tilde \Omega(t))$ is $\mathrm{SO}(n)$-diffeomorphic to $ \lc \tilde\Phi_i \vert \tilde \Sigma_i(t)\rc \times [0,1]$. In particular, $\tilde \Sigma_i(t)$ is $\mathrm{SO}(n)$-diffeomorphic to $\tilde \Sigma_i(s)$.

Now, let $\Sigma_i(t):=\rho_i(\tilde \Sigma_i(t))$. It follows that $\Sigma_i(t)$ is diffeomorphic to $\Sigma_i(s)$. In other words, $\Sigma_i(t)$ for $t \in [s/2,s]$ forms a foliation of the end of $U_i(s)$. On the other hand, from Lemma \ref{L203} (a), $\partial \Omega_i(t)$ for $t \in [s/2,s]$ forms another foliation of the end of $U_i(s) \cap \Omega_i(s)$. 

By a standard topological argument (see, e.g., \cite[Theorem 5.23]{CL21}), it follows that $\partial U_i(s)$ is $h$-cobordant to $\partial \Omega_i(s)$. From the $h$-cobordism, it is straightforward to construct a homotopy equivalence from $U_i(s)$ to $\Omega_i(s)$ and, by Lemma \ref{L203} (a), to $M_i$.
\end{proof}

Under additional topological assumptions, we have the following conclusions. Statement (a) is well-known to experts; since we could not locate a reference, we provide a proof. Statements (b) and (c) are due to Rong \cite{Rong96}, but we include a proof for the reader's convenience. Notice that any Ricci shrinker has finite fundamental group; see \cite{WW07}.

\begin{prop} \label{prop:topo4}
For fixed $s \ge L(n, A)$, the following conclusions hold:
\begin{enumerate}[label=\textnormal{(\alph{*})}]
\item The fiber of $\tilde \Phi_i$ is a nilmanifold, i.e., $N/\Gamma$, where $N$ is a simply connected nilpotent group and $\Gamma$ is a lattice contained in $N$.

\item The fiber of $\tilde \Phi_i$ is a torus $T^k$.

\item If $\pi_1(FM_i)=0$, then $\tilde \Phi_i: \tilde U_i(s) \to \tilde \Omega(s)$ is a $T^k$-principal bundle such that the $T^k$-action on $\tilde U_i(s)$ commutes with the $\mathrm{SO}(n)$-action. In particular, there exists a $T^k$-action on $U_i(s)$, whose orbits correspond to the singular fibers of $\Phi_i$.
\end{enumerate}
\end{prop}

\begin{proof}
(a): Fix $y \in \tilde \Omega(s/2)$ such that $x= \rho(y)$ is a regular point of $X$. Let $S_i=\Phi_i^{-1}(x)$. By Lemma \ref{lem:sinfi}, $S_i$ has a uniformly bounded second fundamental form and a uniform positive lower bound $\delta _0$ for its normal injectivity radius. Moreover, $S_i$ is an almost flat manifold with bounded curvature and a diameter that tends to zero.

Consider a sequence $\lambda_i \to +\infty$ such that the diameter of $S_i$, under the rescaled metric $h_i:=\lambda^2_i g_i$, converges to zero. Moreover, under this new metric $h_i$, the curvature bound converges to zero, the second fundamental form of $S_i$ approaches zero, and the normal injectivity radius of $S_i$ tends to infinity.

Define $V_i:=B_{g_i}(S_i, \delta _0)$ with fundamental group $\Gamma_i$ and let $V'_i$ be its universal cover. Using the pullback metric $h'_i$ on $V'_i$ and fixing a point $p_i \in S_i$ with lift $p'_i \in V'_i$, we conclude from \cite{BK81} and \cite[Lemma 2.5]{NT08} that the injectivity radius at $p_i'$ tends to infinity. Therefore, we have the following convergence:
\begin{align*}
	( V'_i, h'_i, p'_i, \Gamma_i) \longright{\text{equivariant pointed }C^{\infty}} (\R^n, g_{E}, 0, \Gamma).
\end{align*}

If $m=\mathrm{dim}(X)$, our construction ensures that the preimage of $S_i$ in $V'_i$, denoted by $S'_i$, converges smoothly to the subspace $\R^{n-m} \subset \R^n$. Moreover, the limit group $\Gamma \le \mathrm{Aff}(\R^{n-m})$. Following the same argument as in \cite[Section 8]{Fu90} (see also \cite[Lemma 5.3.3]{Rong10}), we deduce that $\Gamma=H \ltimes \R^{n-m}$, where $H$ is a finite subgroup of $\mathrm{SO}(n-m)$.

Since $\Gamma_i$ acts naturally on $F V'_i$ and $F V'_i/\Gamma_i=FV_i$, we have the following convergence:
\begin{align*}
	\lc FV_i, \tilde h_i ,\tilde p_i,\mathrm{SO}(n) \rc \longright{\text{equivariant pointed Gromov-Hausdorff}} \lc F(\R^n)/\Gamma, \mathrm{SO}(n) \rc,
\end{align*}
where $\tilde p_i$ is a lift of $p_i$, $\tilde h_i=\lambda_i^2 \tilde g_i$ and $F(\R^n)=\R^n \times \mathrm{SO}(n)$ is equipped with the frame bundle metric. The limit $F(\mathbb{R}^n) / \Gamma $, determined by the structure of $\Gamma$, is $\R^m \times (\mathrm{SO}(n)/H)$. By the $\mathrm{SO}(n)$-equivariance of $\tilde \Phi_i$, it follows that $\tilde S_i/H$ is diffeomorphic to $S_i$, where $\tilde S_i:=\tilde \Phi_i^{-1}(y)$. Consequently, from the proof of the theorem on almost flat manifolds (see \cite[Section 8,9]{Fu90}, \cite[Section 5.3]{Rong10} and \cite{BK81}), we conclude that $\tilde S_i$ is a nilmanifold.

(b): Assume $n \ge 3$, so $\pi_1(\mathrm{SO}(n))=\Z_2$. Since $\pi_1(M_i)$ is finite, Proposition \ref{prop:topo3} implies that $\pi_1(U_i(s))$ is also finite. Since $\tilde U_i(s)=FU_i(s)$, the homotopy exact sequence 
\begin{align} \label{eq:301}
\pi_2(U_i(s)) \longrightarrow \pi_1(\mathrm{SO}(n))=\Z_2 \longrightarrow \pi_1(\tilde U_i(s)) \longrightarrow \pi_1(U_i(s))\longrightarrow 0
\end{align}
shows that $\pi_1(\tilde U_i(s))$ is finite. From the fibration $\tilde \Phi_i$, we have another homotopy exact sequence:
\begin{align}\label{eq:302}
\pi_2(\tilde \Omega(s)) \longrightarrow \pi_1(N/\Gamma)=\Gamma \longrightarrow \pi_1(\tilde U_i(s)) \longrightarrow \pi_1(\tilde \Omega(s)) \longrightarrow 0.
\end{align}
This implies that $\Gamma$ has an abelian subgroup $\Gamma'$ of finite index. Since $\Gamma'$ is also a uniform discrete subgroup of $N$, it follows from Malcev's rigidity theorem \cite{Ma51} that $N$ itself is abelian. Consequently, $\Gamma$ is abelian, and $N/\Gamma$ is a torus $T^k$.

(c): By assumption, $\tilde U_i(s)$ is simply connected, as it is homotopy equivalent to $FM_i$. Then, from \eqref{eq:301}, $U_i(s)$ and consequently $M_i$ are also simply connected. Moreover, \eqref{eq:302} implies that $\tilde \Omega(s)$ and hence $Y$ are simply connected.

From statement (b) above and Theorem \ref{thm:fib} (c), the structure group of $\tilde \Phi_i$ is contained in $T^k \rtimes \mathrm{Aut}(\Gamma)$, where $\Gamma=\Z^k$. Since $\Gamma$ is discrete and $\tilde \Omega(s)$ is simply connected, the structure group reduces to $T^k$. In other words, $\tilde \Phi_i$ is a $T^k$-principal bundle.

For any $z_0 \in \tilde U_i(s)$ and $\xi _0 \in \so$, there exist local trivializations $ T^k \times U$ and $ T^k \times V$ around $z_0$ and $\xi _0 z_0$, respectively, such that
\begin{align*}
\xi (t, x)=(\phi_\xi (t,x), \psi_\xi (x))
\end{align*}
for $(t, x) \in T^k \times U$ in a small neighborhood of $z_0$ and $\xi $ in a small neighborhood of $\xi _0$. Note that $\psi_\xi $ is independent of $t$ because $\tilde \Phi_i$ is $\so$-invariant. Since the $\so$-action preserves the flat connections by Theorem \ref{thm:fib} (c), $\phi_\xi $ is an affine transformation, meaning $\phi_\xi  \in T^k \rtimes \mathrm{Aut}(\Gamma)$. As $\mathrm{Aut}(\Gamma)$ is discrete and $\so$ is connected, a continuity argument shows that the $\mathrm{Aut}(\Gamma)$-component of $\phi_\xi $ is the identity. Thus, the $T^k$-action and $\so$-action commute.

Consequently, the $T^k$-action on $\tilde U_i(s)$ descends to a $T^k$-action on $U_i(s)$, whose orbits correspond to the singular fibers of $\Phi_i$.
\end{proof}

For the rest of this section, we further assume that each $M_i$ is simply connected. By the homotopy exact sequence of the fibration $\so \to FM_i \to M_i$, it follows that $\pi_1(FM_i)=0$ or $\Z_2$. 

If $\pi_1(FM_i)=0$, the assumption of Proposition \ref{prop:topo4} (c) is satisfied. Thus, there exists a free $T^k$-action on $\tilde U_i(s)$ which commutes with the $\so$-action. 

If $\pi_1(FM_i)=\Z_2$, then $M_i$ is a spin manifold. Indeed, this can be derived from the exact sequence: 
\begin{align*}
H^1(M, \Z_2)=0 \longrightarrow H^1(FM, \Z_2)=\Z_2 \longrightarrow H^1(\so,\Z_2)=\Z_2 \longright{w_2} H^2(M, \Z_2),
\end{align*}
where $w_2$ corresponds to the second Stiefel-Whitney class.

Let $\hat FM_i$ denote the universal cover of $FM_i$. The space $\hat FM_i$ naturally serves as the total space of the $\spin$-bundle over $M_i$, where the covering map $\pi_i: \hat FM_i \to FM_i$, when restricted to the fiber, corresponds to the double cover $\spin \to \so$. Let $\hat p_i$, $\hat g_i$ and $\hat f_i$ denote the lifts of $p_i$, $g_i$ and $f_i$, respectively. Analogous to Proposition \ref{prop:frame}, we have
\begin{align} \label{eq:conv4}
	(\hat FM_i, \hat g_i, \spin, \hat f_i ,\hat p_i) \longright{\text{equivariant pointed Gromov-Hausdorff}} (\hat Y, g_{\hat Y}, \spin, f_{\hat Y},p_{\hat Y}),
\end{align}
where $(\hat Y,g_{\hat Y})$ is a smooth Riemannian manifold, and $f_{\hat Y}$ is a smooth function. Furthermore, the free $\Z_2$-action on $\hat FM_i$ converges to a free $\Z_2$-action on $\hat Y$, inducing the projection $ \pi: \hat Y \to \hat Y/\Z_2=Y$. Let $\hat \Omega(s):=\{ x \in \hat Y \mid f_{\hat Y}(x) \le s\}$. Clearly, $\hat \Omega(s)$ is a double cover of $\tilde \Omega(s)$.

We can lift the fibration $\tilde \Phi_i: \tilde U_i(s) \to \tilde \Omega(s)$ to a fibration $\hat \Phi_i: \hat U_i(s) \to \hat \Omega(s)$, where $\hat U_i(s) \subset \hat FM_i$ is a $\spin$-invariant double cover of $\tilde U_i(s)$. This setup results in the following commutative diagram:
\[\begin{tikzcd} 
{(\hat U_i(s),\hat{g}_i)} && {(\hat \Omega(s),d_{g_{\hat Y}})} \\
\\
{(\tilde U_i(s),\tilde g_i)} && {(\tilde \Omega(s),d_{g_Y})}
\arrow["\tilde\Phi_i", from=3-1, to=3-3]
\arrow["{ \pi}", from=1-3, to=3-3]
\arrow["{\pi_i}", from=1-1, to=3-1]
\arrow["\hat \Phi_i", from=1-1, to=1-3]
\end{tikzcd}\]

By following the proof of Proposition \ref{prop:topo4} (c), we obtain the following result:

\begin{prop} \label{prop:topo5}
For fixed $s \ge L(n, A)$, $\hat \Phi_i: \hat U_i(s) \to \hat \Omega(s)$ is a $T^k$-principal bundle so that the $T^k$-action on $\hat U_i(s)$ commutes with the $\spin$-action. In particular, there exists a $T^k$-action on $U_i(s)$, whose orbits correspond to the singular fibers of $\Phi_i$.
\end{prop}

Combining Proposition \ref{prop:topo4} with Proposition \ref{prop:topo5}, we conclude that, under the assumption $\pi_1(M_i)=0$, there always exists a $T^k$-action on $U_i(s)$, whose orbits correspond to the singular fibers of $\Phi_i$. By our construction, the $T^k$-action has no fixed set if $i$ is sufficiently large. Indeed, if there exists a $T^k$-orbit that consists of a single point $q_i \in U_i(s)$, it follows from Lemma \ref{lem:sinfi} that the injectivity radius at $q_i$ is uniformly bounded below by a positive constant, which contradicts our assumption of collapsing.

In general, the metric $g_i$ or function $f_i$ may not be $T^k$-invariant. Using a standard averaging argument as in \cite{CFG92}, we can construct a nearby $T^k$-invariant metric and function. Specifically, if $FM_i$ is simply connected, we can first average $\tilde g_i$ and $\tilde f_i$ on $\tilde U_i(s)$ using the free $T^k$-action to obtain $\tilde g'_i$ and $\tilde f'_i$. The difference between $\tilde g_i$ (resp. $\tilde f_i$) and $\tilde g'_i$ (resp. $\tilde f'_i$) in any $C^l$-norm converges to zero as $i \to \infty$. (Details of this calculation can be found, for example, in \cite[Proposition 5.10]{CL21}.)

Since the $T^k$-action and the $\so$-action commute by Proposition \ref{prop:topo4} (c), both $\tilde g'_i$ and $\tilde f'_i$ are $\so$-invariant. Consequently, their descents on $U_i(s)$, denoted by $g_i'$ and $f_i'$, are a $T^k$-invariant metric and a $T^k$-invariant function, respectively. Moreover, by construction, $(\tilde{U}_i(s), \tilde g'_i)$ is exactly the frame bundle associated with $(U_i(s), g'_i)$.

If $\pi_1(FM_i)=\Z_2$, a similar construction can be applied using Proposition \ref{prop:topo5}.

In summary, we have the following result:

\begin{prop} \label{prop:ave}
Given $s \ge L(n, A)$ and $\ep>0$, if $i$ is sufficiently large, there exist $T^k$-invariant smooth metric $g^{\ep}_i$ and function $f^{\ep}_i$ on $U_i(s)$ such that for any $0 \le l \le [\ep^{-1}]$, 
\begin{align*}
|\na_{g_i}^l (g^{\ep}_i-g_i) |_{g_i}+|\na_{g_i}^l (f^{\ep}_i-f_i)|_{g_i} \le \ep.
\end{align*}
\end{prop}

We conclude this section with the following result, based on prior results after taking a diagonal subsequence. This theorem will be used in the next section.

\begin{thm} \label{thm:dia}
$(M^n_i, g_i, f_i, p_i)$ be a sequence of Ricci shrinkers with $|\Rm(g_i)|\le A$ such that $\pi_1(M_i)=0$ and $\mmu( g_i) \to -\infty$. Then, by taking a subsequence if necessary, the following statements hold:
\begin{enumerate}[label=\textnormal{(\alph{*})}]
\item There exists a domain $U_i \subset M_i$ with an effective $T^k$-action. Moreover, $U_i$ is homotopy equivalent to $M_i$.

\item There exist $T^k$-invariant metric $g_i'$ and function $f_i'$ on $U_i$ such that
\begin{align*}
\sum_{l=0}^{i^{2}} \lc  |\na_{g_i}^l (g'_i- g_i) |_{g_i}+|\na_{g_i}^l (f'_i- f_i)|_{ g_i} \rc \le i^{-2}.
\end{align*}

\item We have
\begin{align} \label{eq:conv5}
	(U_i, g'_i,f'_i , p_i) \longright{\text{pointed Gromov-Hausdorff}} (X, d_X, f_X,p_X),
\end{align}
where $(X, d_X)$ is a simply connected geodesic space satisfying Proposition \ref{prop:po2}. Moreover, $X=\bigcup_i \Omega^X_i$, where $\Omega^X_i \subset \Omega^X_{i+1}$, and each $\Omega^X_i$ is homeomorphic to $X$.

\item There exists a singular fibration $\Phi_i: U_i \to \Omega^X_i$ such that the singular fibers are the orbits of the $T^k$-action. Moreover, given $r \ge 1$, there exists a constant $C>1$ independent of $i$ such that any fiber of $\Phi_i$, which is contained in $\{f_i' \le r^2\}$, has the second fundamental form bounded by $C$ and normal injectivity radius bounded below by $C^{-1}$.

\item There exists a simply connected smooth manifold $F_i$, which is the $\mathrm{SO}(n)$ or $\mathrm{Spin}(n)$ bundle of $(U_i, g_i')$. Moreover, $F_i$ admits a $(G \times T^k)$-action, where $G=\so$ or $\spin$, such that both the sub-$T^k$-action and sub-$G$-action are both free. Additionally, there exists a canonical $G \times T^k$-invariant metric $\tilde g_i'$ and function $\tilde f_i'$ on $F_i$, such that $f_i' \circ \rho_i=\tilde f_i'$, where $\rho_i: (F_i,\tilde g_i') \to (U_i ,g_i')$ is a Riemannian submersion, and the induced fiber metric is a fixed standard bi-invariant metric on $G$.

\item We have
\begin{align} \label{eq:conv6}
	(F_i, \tilde g'_i, G,\tilde f'_i ,\tilde p'_i) \longright{\text{equivariant pointed Gromov-Hausdorff}} (Y, g_Y, G, f_Y,p_Y),
\end{align}
where $(Y, g_Y)$ is a smooth, simply connected, Riemannian manifold such that $(Y,g_Y)/G$ is isometric to $(X, d_X)$. Additionally, $Y=\bigcup_i \Omega^Y_i$, where $\Omega^Y_i \subset \Omega^Y_{i+1}$, and each $\Omega^Y_i$ is $G$-diffeomorphic to $Y$.

\item There exists a $G$-equivariant fibration $\tilde \Phi_i: F_i \to \Omega^Y_i$ that coincides with the principal bundle induced by the $T^k$-action. Furthermore,
\begin{enumerate}[label=\textnormal{(\roman{*})}]
\item For any horizontal vector $\zeta$, 
\begin{align*}
(1-i^{-2}) |\zeta| \le |\mathrm{d}\tilde \Phi_i(\zeta)| \le (1+i^{-2}) |\zeta|.
\end{align*}

\item Each fiber of $\tilde \Phi_i$ has diameter bounded by $i^{-2}$. 

\item Given $r \ge 1$, there exists a constant $C>1$ independent of $i$ such that any fiber of $\tilde \Phi_i$, which is contained in $\{\tilde f_i' \le r^2\}$, has the second fundamental form bounded by $C$ and normal injectivity radius bounded below by $C^{-1}$.
\end{enumerate}

\item The following commutative diagram holds:
\[\begin{tikzcd} 
{(F_i, \tilde g'_i)} && {( \Omega^Y_i, g_Y)} \\
\\
{(U_i, g'_i)} && {(\Omega^X_i, d_X)}
\arrow["\Phi_i", from=3-1, to=3-3]
\arrow["{\rho}", from=1-3, to=3-3]
\arrow["{\rho_i}", from=1-1, to=3-1]
\arrow["\tilde \Phi_i", from=1-1, to=1-3]
\end{tikzcd}\]
\end{enumerate}
\end{thm}

The limit space in \eqref{eq:conv5} coincides with the limit space in \eqref{eq:conv1}. Notice that by taking a subsequence, for some $s_i \to \infty$, $F_i$ in proposition \ref{thm:dia} (e) is either $\tilde{U}_i(s_i)$ or $\hat U_i(s_i)$, depending on whether $\pi_1(FM_i)=0$ or $\pi_1(FM_i)=\Z_2$. In the first case, $G=\so$, and the limit space in \eqref{eq:conv6} coincides with the limit space in \eqref{eq:conv2}. In the second case, $G=\spin$, and the limit space in \eqref{eq:conv6} coincides with the limit space in \eqref{eq:conv4}. The fact that $Y$ and $X$ are simply connected follows from \cite[Chapter II, Corollary 6.3]{Bre72}.

\section{Proof of the main theorem}

The main goal of this section is to prove Theorem \ref{thm:001}. Once Theorem \ref{thm:001} is established, Corollary \ref{cor:001} follows directly from Proposition \ref{prop:smooth}. 

The strategy is based on a contradiction argument. Let $(M_i, g_i, f_i, p_i) \in \mathcal E(n, A)$ be a sequence of Ricci shrinkers such that each $M_i$ is simply connected and has a finite second homotopy group. We aim to show that, under the assumption $\mmu(g_i) \to -\infty$, we can derive a contradiction.

Since each $M_i$ is simply connected, we assume the sequence satisfies the conclusions of Theorem \ref{thm:dia}. In particular, the objects $U_i$, $g_i'$, $f_i'$, $F_i$, $\tilde g_i'$, $\tilde f_i'$, and others are defined.

The additional assumption of a finite second homotopy group imposes strong restrictions on the topology of this sequence. Specifically, we have (see Definition \ref{dfn:205}):

\begin{prop} \label{prop:topo401}
For any $i<j$, there exists a weakly $(G \times T^k)$-diffeomorphism $\tilde h: F_i \to F_j$ that is $G$-equivariant. Moreover, the following diagram commutes:
\[\begin{tikzcd}
{F_i} && {F_j} \\
{ \Omega^Y_i} & Y & {\Omega^Y_j}
\arrow["{\tilde h}", from=1-1, to=1-3]
\arrow["{\tilde \Phi_i}", from=1-1, to=2-1]
\arrow["{\tilde \Phi_j}", from=1-3, to=2-3]
\arrow["{\tau_i}", from=2-1, to=2-2]
\arrow["{\tau_j}"', from=2-3, to=2-2]
\end{tikzcd}\]
Here, $\tau_i$ and $\tau_j$ are $G$-diffeomorphisms obtained from Theorem \ref{thm:dia} (f).
\end{prop}

\begin{proof}
By Theorem \ref{thm:dia} (a), $U_i$ is simply connected and has a finite second homotopy group. From Theorem \ref{thm:dia} (e), we know there is a fiber bundle $G \to F_i \to U_i$. Using the homotopy exact sequence, we obtain:
\begin{align*}
\pi_2(G)=0 \longrightarrow \pi_2(F_i) \longrightarrow \pi_2( U_i) \longrightarrow \pi_1(G),
\end{align*}
where the first equality holds because $G=\so$ or $\spin$, and $\pi_2(\so)=\pi_2(\spin)=0$. Consequently, $\pi_2(F_i)$ is finite. The same conclusion holds for $F_j$.

By Theorem \ref{thm:dia} (f) and (g), $F_i/T^k$ (resp. $F_j/T^k$) is $G$-diffeomorphic to $\Omega^Y_i$ (resp. $\Omega^Y_j$) through the map $\tilde \Phi_i$ (resp. $\tilde \Phi_j$). Since both $\Omega^Y_i$ and $\Omega^Y_j$ are $G$-diffeomorphic to $Y$ via $\tau_i$ and $\tau_j$, respectively, the proof follows from \cite[Key Lemma 2.6]{PT99}.

For the reader's convenience, we provide a sketch of the proof. The homotopy exact sequence of the bundle of $T^k \to F_i \to Y$ gives:
\begin{align}\label{eq:402}
\pi_2(F_i) \longrightarrow \pi_2(Y) \longright{e} \pi_1(T^k) \longrightarrow \pi_1(F_i)=0.
\end{align}
Since $Y$ is simply connected, all principal $T^k$-bundles over $Y$ are classified by elements in $H^2(Y,\Z^k)=\mathrm{Hom}(\pi_2(Y)/\mathrm{tor}, \Z^k)$. The map $e$ in \eqref{eq:402} (called the Euler class) is an element of $\mathrm{GL}(\Z, k)$ and determines the principal bundle. In other words, any two principal $T^k$-bundles over $Y$ differ by an automorphism of $T^k$ (see also \cite[Lemma 3.4]{FR99}). 

Consequently, we obtain the following commutative diagram:
\[\begin{tikzcd} 
{(F_i, T^k)} && {( F_j, T^k)} \\
\\
{(\Omega^Y_i, G)} && {(\Omega^Y_j, G)}
\arrow["h", from=3-1, to=3-3]
\arrow["{\tilde \Phi_j}", from=1-3, to=3-3]
\arrow["{\tilde \Phi_i}", from=1-1, to=3-1]
\arrow["h^*", from=1-1, to=1-3]
\end{tikzcd}\]
Here, $h=\tau_j^{-1} \circ \tau_i$, and $h^*$ is a weakly $T^k$-diffeomorphism. It is straightforward to show that for any $x \in F_i$ and $\xi \in G$, both $h^*(\xi x)$ and $\xi h^*(x)$ lie in the same $T^k$-orbit in $F_j$. Thus, we can define a map $\eta: G \times F_i \to T^k$ such that 
\begin{align*}
h^*(\xi x)=\eta(\xi ,x)\cdot \xi h^*(x).
\end{align*}

For any $x \in F_i$, let $\omega(x)$ denote the mean value of $\eta(\xi , x)$ over $\xi \in G$. It is easy to verify that $\omega(x)$ is $T^k$-invariant. Defining $\tilde h(x)=\omega(x) h^*(x)$, we obtain a weakly $(G \times T^k)$-diffeomorphism from $F_i$ to $F_j$ that is also $G$-equivariant.
\end{proof}

As a consequence of Proposition \ref{prop:topo401}, there exists a smooth $(G \times T^k)$-manifold $F$, a principal $T^k$-bundle $\tilde \Phi: F \to Y$ corresponding to the $T^k$-action, and a weak $(G \times T^k)$-diffeomorphism $\phi_i: F \to F_i$ that is $G$-invariant such that the following diagram commutes:
\[\begin{tikzcd} 
{F} && {F_i} \\
\\
{Y} && {\Omega^Y_i}
\arrow["\tau_i", from=3-3, to=3-1]
\arrow["{\tilde \Phi_i}", from=1-3, to=3-3]
\arrow["{\tilde \Phi}", from=1-1, to=3-1]
\arrow["\phi_i", from=1-1, to=1-3]
\end{tikzcd}\]

We set $U=F/G$, and the map $\phi_i$ induces the following commutative diagram:
\[\begin{tikzcd} 
{F} && {F_i} \\
\\
{U} && {U_i}
\arrow["\bar \phi_i", from=3-1, to=3-3]
\arrow["{\mathrm{mod}\ G}", from=1-3, to=3-3]
\arrow["{\mathrm{mod}\ G}", from=1-1, to=3-1]
\arrow["\phi_i", from=1-1, to=1-3]
\end{tikzcd}\]
Here, $\bar \phi_i$ is the map induced by $\phi_i$. For later applications, we define
\begin{align*} 
\bar g_i:=\bar \phi_i^* g'_i, \quad \bar f_i:=\bar \phi_i^* f'_i,\quad \tilde g_i:=\phi_i^* \tilde g'_i \quad \text{and} \quad \tilde f_i:=\phi_i^* \tilde f'_i.
\end{align*}
In the following, we will only consider the fixed $(G \times T^k)$-action on $F$. Note that both $\tilde g_i$ and $\tilde f_i$ are invariant under this $(G \times T^k)$-action. Additionally, $\bar g_i$ and $\bar f_i$ are invariant under the $T^k$-action induced on $U$.

We fix a locally finite open cover $\{V_{\alpha}\}$ of $Y$ such that:
\begin{enumerate}
\item $V_{\alpha}$ is a geodesic ball with center $y_{\alpha}$ and radius $r_\alpha $.

\item $V_{\alpha}$ is convex with compact closure.

\item There exists a trivialization of $\tilde \Phi: F \to Y$ given by
\begin{align} \label{eq:402xxa}
\tilde \Phi^{-1}(V_{\alpha}) \longright{\varphi_{\alpha}} T^k \times V_{\alpha},
\end{align}
such that $\varphi_{\alpha}(tx)=t \varphi_{\alpha}(x)$ for any $x \in \tilde \Phi^{-1}(V_{\alpha})$ and $t \in T^k$. 
\end{enumerate}

In the following, we may shrink $\{V_{\alpha}\}$ when necessary, ensuring that the new cover still satisfies the above conditions. This process will not be explicitly mentioned if there is no confusion.

The transition map $\phi_{\beta \alpha}:=\varphi_{\beta} \circ \varphi^{-1}_{\alpha}$, defined on $T^k \times (V_{\alpha} \cap V_{\beta})$, is given by
\begin{align} \label{eq:402a}
\phi_{\beta \alpha}(t, x)=(s_{\beta \alpha}(x) t, x),
\end{align}
for any $(t,x) \in T^k \times (V_{\alpha} \cap V_{\beta})$, where $s_{\beta \alpha}(x): V_\alpha \cap V_\beta \to T^{k}$ satisfies $s_{\alpha \alpha } = \mathrm{id},$ and $s_{\alpha \gamma }s_{\gamma \beta }s_{\beta \alpha } = \mathrm{id}$ on $V_\alpha \cap V_\beta \cap V_\gamma $. 

Now, we fix a decomposition 
\begin{align} \label{eq:402xa}
T^k=S^1 \times T^{k-1}
\end{align}
and write $t=(t',t'') \in S^1 \times T^{k-1}$ for $t \in T^k$. Then, \eqref{eq:402a} becomes
\begin{align}\label{eq:402b}
\phi_{\beta \alpha}(t', t'' ,x)=(s'_{\beta \alpha}(x) t', s''_{\beta \alpha}(x) t'', x),
\end{align}
where we decompose $s_{\beta \alpha}=(s'_{\beta \alpha}, s''_{\beta \alpha})$. 

Using this, we can glue $\{S^1 \times V_{\alpha}\}$ with transition maps:
\begin{align} \label{eq:403}
\phi'_{\beta \alpha}(t', x)=(s'_{\beta \alpha}(x) t', x),
\end{align}
to form a principal $S^1$-bundle, which is precisely the quotient of $\tilde \Phi$ from $F/T^{k-1}$ to $Y$.

If we denote the quotient metric of $\tilde g_i$ and the quotient function of $\tilde f_i$ on $F/T^{k-1}$ by $ \tilde{g}_{i,k-1}$ and $  \tilde{f}_{i,k-1}$, respectively, we obtain the following result:

\begin{lem}\label{L401}
Given $L>1$, there exists a constant $C>1$, independent of $i$, such that
\begin{align}\label{eq:401aa}
\mathrm{Length}_{ \tilde{g}_{i,k-1}}(S^1) \le C \mathrm{diam}_{ \tilde{g}_{i,k-1}}(S^1),
\end{align}
where $S^1$ is any circle fiber of $F/T^{k-1} \to Y$ contained in $\{\tilde{f}_{i,k-1} \le L\}$.
\end{lem}

\begin{proof}
By Theorem \ref{thm:dia} (g), any $T^k$-fiber of $\tilde \Phi$, contained in $\{\tilde f_i \le L\}$, has a second fundamental form uniformly bounded above by a constant $C>1$ and a normal injectivity radius bounded below by $C^{-1}$, where $C$ is independent of $i$. 

Taking the quotient by $T^{k-1}$, each $S^1$-fiber of $F/T^{k-1} \to Y$, contained in $\{ \tilde{f}_{i,k-1} \le L\}$, inherits these properties: the second fundamental form remains uniformly bounded above by $C$, and the normal injectivity radius stays bounded below by $C^{-1}$.

As a result, the intrinsic metric on each circle fiber, induced by $\tilde{g}_{i,k-1}$, is uniformly comparable to $\tilde{g}_{i,k-1}$. Thus, the length of the circle fiber with respect to $\tilde{g}_{i,k-1}$ is uniformly comparable to its diameter with respect to $\tilde{g}_{i,k-1}$. This establishes \eqref{eq:401aa}.
\end{proof}

\begin{lem}\label{L401a}
For a given $\alpha$, there exists a sequence of $\R^k$-equivariant diffeomorphisms $\varphi_i: \R^k \times V_{\alpha} \to \R^k \times V_{\alpha}$ such that 
\begin{align*}
\varphi_i^* (\tilde{g}_i^{\alpha}) \longright{C^{\infty}} \tilde{h}^{\alpha},
\end{align*}
where $\tilde{g}_i^{\alpha}$ is the lift of $\tilde g_i$ to $\R^k \times V_{\alpha}$. Moreover, the metric $\tilde{h}^{\alpha}$ is invariant under the $\R^k$-action. 
\end{lem}

\begin{proof}
By construction, $\tilde{g}_i^{\alpha}$ has uniformly bounded curvature as well as uniformly bounded higher-order derivatives of the curvature. Furthermore, it follows from Theorem \ref{thm:dia} (g) (iii) that, with respect to $\tilde{g}_i^{\alpha}$, the $\R^k$-fiber over $y_{\alpha}$, the center of $V_{\alpha}$, has a uniformly bounded second fundamental form, and the normal injectivity radius is bounded below by a uniform positive constant. Consequently, it is straightforward to show (see, e.g., \cite[Lemma 2.5]{NT08}) that the injectivity radius of $\tilde{g}_i^{\alpha}$ at the point $y'_{\alpha}$, a lift of $y_{\alpha}$, has a uniform positive lower bound independent of $i$.

Using standard arguments in convergence theory, and by shrinking $V_{\alpha}$ if necessary, there exists a self-diffeomorphism $\varphi_i:V_{\alpha} \to V_{\alpha}$ such that 
\begin{align*}
\varphi_i^* \lc \tilde{g}_i^{\alpha} \vert_{\vec{0} \times V_{\alpha}} \rc
\end{align*}
converges smoothly to a metric tensor $\tilde{h}^{\alpha}$ on $V_{\alpha}$. This map $\varphi_i$ can then be extended to a diffeomorphism from $\R^k \times V_{\alpha}$ to $\R^k \times V_{\alpha}$ in an $\R^k$-equivariant manner. 

Since each $\tilde{g}_i^{\alpha}$ is $\R^k$-invariant, it follows immediately that $\varphi_i^*(\tilde{g}_i^{\alpha})$ converges smoothly to an $\R^k$-invariant metric $\tilde{h}^{\alpha}$.
\end{proof}

In general, we cannot directly glue all $( \R^k \times V_{\alpha}, \tilde{h}^{\alpha})$ together due to the complexity of the obstructions. To overcome this, we consider the decomposition $\R^k \times V_{\alpha}=\R \times \R^{k-1} \times V_{\alpha}$ corresponding to \eqref{eq:402xa}. By taking the quotient by $\R^{k-1}$, we denote the resulting quotient metric of $\tilde{h}^{\alpha}$ on $\R \times V_{\alpha}$ by $\bar h^{\alpha}$. This formulation allows us to glue all $( \R \times V_{\alpha}, \bar h^{\alpha})$ together in the following way. Essentially, this process is equivalent to unwrapping the circle fibers in the fibration $F/T^{k-1} \to Y$. 

Let $\bar g_i^{\alpha}$ on $\R \times V_{\alpha}$ be the quotient metric of $\tilde{g}_i^{\alpha}$ by $\R^{k-1}$. While $\bar g_i^{\alpha}$ may not converge smoothly to $\bar h^{\alpha}$ due to the lack of curvature control for $\bar g_i^{\alpha}$, Lemma \ref{L401a} ensures the following weaker conclusion, which suffices for our purposes:
\begin{align} \label{eq:404a}
	(\R \times V_{\alpha} , \bar g_i^{\alpha}) \longright{\text{local Gromov-Hausdorff}} (\R \times V_{\alpha} , \bar h^{\alpha}).
\end{align}
Here, the local Gromov-Hausdorff convergence means that for any $y \in V_{\alpha}$, there exists $\ep>0$ such that for any $r \in \R$, the ball $B_{\bar h^{\alpha}}((r,y), \ep) \subset \R \times V_{\alpha}$, and 
\begin{align} \label{eq:404ax}
	\lc B_{\bar g_i^{\alpha}}((r,y), \ep), d_{\bar g_i^{\alpha}} \rc \longright{\text{Gromov-Hausdorff}} \lc B_{\bar h^{\alpha}}((r,y), \ep), d_{\bar h^{\alpha}} \rc,
\end{align}
where $d_{\bar g_i^{\alpha}}$ and $d_{\bar h^{\alpha}}$ are the distance functions of the length structures induced by $\bar g_i^{\alpha}$ and $\bar h^{\alpha}$, respectively.

Next, we identify the obstruction to unwrapping the circle fibers. First, we lift the transition map in \eqref{eq:403} to $\tilde \phi_{\beta \alpha}: \R \times (V_{\alpha} \cap V_{\beta}) \to \R \times (V_{\alpha} \cap V_{\beta})$ by
\begin{align*} 
\tilde \phi_{\beta \alpha}(r, x)=(r+\tilde s_{\beta \alpha}(x), x),
\end{align*}
for $(r, x) \in \R \times (V_{\alpha} \cap V_{\beta})$, where $\tilde s_{\beta \alpha}(x)$ is fixed smooth lift of $s'_{\beta \alpha}(x)$. Note that the lift is not unique, as any two lifts differ by a translation in $\Z$. 

For $V_{\alpha} \cap V_{\beta} \cap V_{\gamma} \ne \emptyset$, we define 
\begin{align} \label{eq:405}
\lambda_{\alpha \beta \gamma}:=\tilde s_{\alpha \gamma} + \tilde s_{\gamma \beta }+ \tilde s_{\beta \alpha}.
\end{align}
The value $\lambda_{\alpha \beta \gamma}(x)$, for $x \in V_{\alpha} \cap V_{\beta} \cap V_{\gamma}$, is a constant integer. The collection of all such $\{\lambda_{\alpha \beta \gamma}\}$ constitutes the Euler class of the fibration $F/T^{k-1} \to Y$ in the second \v{C}ech cohomology group $H^2(Y,\Z)$. 

The effect of the obstruction \eqref{eq:405}, under this metric $\bar g_i^{\alpha}$, is given by
\begin{align} \label{eq:406}
d_{\bar g_i^{\alpha}} \lc (r, x), (r+\lambda_{\alpha \beta \gamma},x) \rc \le |\lambda_{\alpha \beta \gamma}| \cdot \mathrm{Length}_i(S^1_x),
\end{align}
where $\mathrm{Length}_i(S^1_x)$ is the length of the circle over $x$ with respect to the quotient metric $\tilde{g}_{i, k-1}$ on $S^1 \times V_{\alpha} \subset F / T^{k-1}$. It follows from Lemma \ref{L401} that $\mathrm{Length}_i(S^1_x)$ tends to zero as the diameter of each $T^k$ with respect to $\tilde{g}_i$ converges to zero locally uniformly. Combining \eqref{eq:404a}, \eqref{eq:404ax}, and \eqref{eq:406}, we conclude that the obstruction vanishes.

Thus, by taking a further subsequence, we can glue all the pieces $(\R \times V_{\alpha}, \bar h^{\alpha})$ to obtain an $\R$-bundle $(Z, g_Z)$ over $Y$. Topologically, $Z$ is diffeomorphic to $\R \times Y$ because $Y$ is simply connected.

Since $g_Z$ is locally identified with $\bar h^{\alpha}$, it is clear that $g_Z$ is invariant under the $\R$-action. Similarly, one can lift $\tilde f_i$ to obtain $\bar f_i^{\alpha}$ on $\R \times V_{\alpha}$, and by taking the limit, we obtain a smooth function $f_Z$, which is also invariant under the $\R$-action.

Next, we show that $Z$ is equipped with a $G$-action that commutes with the $\R$-action. Since the $G$-action descends to $F/T^{k-1}$, there exists a decomposition $G=G_0 \cup \cdots \cup G_l$, so that for each $0\leq j\leq l$, $G_j$ is a small contractible open set, and there exists a refined cover $\{V^j_{\alpha'}\}$ of $\{V_{\alpha}\}$ such that for any $\xi \in G_j$ and $(t, x) \in S^1 \times V^j_{\alpha'}$, $\xi (t, x)$ lies within some $S^1 \times V_{\alpha}$. 

For $\xi \in G_j$ and $(t, x) \in S^1 \times V^j_{\alpha'}$, we can write
\begin{align*}
\xi (t, x)=(\phi_\xi (x) t, \psi_\xi (x)),
\end{align*}
where $\phi_\xi (x) \in S^1$. Using this, we define the action of $G_j$ on $\R \times V^j_{\alpha'}$ as
\begin{align*} 
\xi (r, x)=(\tilde\phi_\xi (x)+r, \psi_\xi (x)),
\end{align*}
for $\xi  \in G_j$ and $(r, x) \in \R \times V^j_{\alpha'}$, where $\tilde \phi_\xi (x)$ is a fixed smooth lift of $\phi_\xi (x)$.

In this way, for each element in $G_j$, we define a map:
\begin{align*}
\coprod_{\alpha'} \R \times V^j_{\alpha'} \to\coprod_{\alpha} \R \times V_{\alpha}.
\end{align*}
In general, these maps cannot always be glued into a global map $Z \to Z$ due to potential inconsistencies across different pieces. However, as in the argument above, it can be shown that, by taking the limit, all discrepancies vanish. Consequently, we obtain an action of $G_j$ on $Z$. 

Repeating this process for each $1 \le j \le l$, we define an action of $G$ on $Z$. Moreover, it is straightforward to verify that this is indeed a group action: since the map on $F / T^{k-1}$ is a group action, any discrepancies arising from the lifting process tend to zero as we take the limit. 

By construction, the $G$-action is isometric with respect to $g_Z$ and commutes with the $\R$-action. Additionally, $f_Z$ is invariant under $G \times \R$.

Now, we define $\hat \iota : (Z,g_Z) \to (Y,g_Y)$, where $\hat \iota $ is a Riemannian submersion. For each $y \in Y$, we denote the $\R$-fiber in $Z$ over $y$ by $\R_y$. From Theorem \ref{thm:dia} (d) and our construction, it is clear that for any $y$ in a compact set of $Y$, $\R_y$ is an embedded submanifold, which has a uniformly bounded second fundamental form and a uniformly positive normal injectivity radius. Furthermore, since $g_Z$ is $\R$-invariant, each $\R_y$, with the induced metric from $g_Z$, is isometric to $(\R,g_E)$, the standard flat metric.

\begin{lem}\label{L402}
The Riemannian manifold $(Z,g_Z)$ is complete.
\end{lem}

\begin{proof}
Suppose $\{z_i\}$ is a Cauchy sequence in $Z$. Since $\hat \iota $ is distance-decreasing, $\{\hat \iota (z_i) \}$ is also a Cauchy sequence in $Y$. Because $Y$ is complete, we can assume without loss of generality that $\hat \iota (z_i) \to y \in Y$. 

Let $z_i'$ be the projection of $z_i$ to $\R_y$, which is well defined and unique since the normal injectivity radius of $\mathbb{R}_y$ is bounded below by a positive constant. It is clear that $\{z_i'\}$ is also a Cauchy sequence with respect to $d_Z$. By taking a subsequence if necessary, we may assume $d_Z(z_i',z_0') \le \ep$ for all $i$, where $\ep$ is a small constant to be determined later. 

Since $\R_y$ is an embedded submanifold, we can choose a small enough number $\epsilon >0$ to ensure that $d_Z$ and $d'$, the induced metric on $\R_y$, are comparable on $B_{Z}(z_0',\ep)$. Therefore, $z_i'$ converges to some $z_{\infty} \in \R_y$ under $d_Z$. Consequently, $z_i \to z_{\infty}$ in $Z$.
\end{proof}

\begin{lem}\label{L403}
For any $y \in Y$, the fiber $\R_y$ is noncompact.
\end{lem}

\begin{proof}
Fix $z_0 \in \R_y$, and define $z_i=t_i  z_0$, where $t_i \in \R$ and $t_i \to +\infty$. If $\R_y$ is compact, we could assume $z_i \to z_{\infty} \in \R_y$ under $d_Z$. However, this would contradict the same reasoning used in the proof of Lemma \ref{L402}. Therefore, $\R_y$ must be noncompact.
\end{proof}

Now, we define $W=Z/G$ with the quotient metric $d_W$. Additionally, let $f_W$ on $W$ denote the quotient function of $f_Y$. Since the $G$-action and $\R$-action commute, the $\R$-action descends naturally to $W$. Moreover, this $\R$-action is isometric, and $f_W$ remains invariant under the $\R$-action. 

Because $G$ is compact, each $\R$-orbit in $W$ is isomorphic to $\R$. Consequently, $W$ can be viewed as an $\R$-bundle over $X$. Since $X$ is simply-connected, $W$ is homeomorphic to $\R \times X$. Let $\iota: W \to X$ be the quotient map. This relationship is illustrated in the following commutative diagram:
\[\begin{tikzcd} 
{(Z, g_Z ,f_Z)} && {(Y, g_Y,f_Y)} \\
\\
{(W,d_W,f_W)} && {(X, d_X, f_X)}
\arrow["\iota ", from=3-1, to=3-3]
\arrow["{\mathrm{mod}\ G}", from=1-3, to=3-3]
\arrow["{\mathrm{mod}\ G}", from=1-1, to=3-1]
\arrow["\hat{\iota }", from=1-1, to=1-3]
\end{tikzcd}\]

\begin{lem}\label{L404}
The space $(W, d_W)$ is complete. Moreover, each $\R$-orbit in $W$ is noncompact.
\end{lem}
\begin{proof}
The first statement follows directly from Lemma \ref{L402}. 

To prove the second, fix a point $w_0 \in W$ with a lift $z_0 \in Z$. Consider a sequence $t_i \to +\infty$. Suppose, for contradiction, that $d_W(t_i w_0, w_0) \le C$ for some constant $C>0$. By the definition of the quotient metric, there exists a sequence $\xi _i \in G$ such that:
\begin{align*}
d_Z(t_i z_0, \xi _i z_0)=d_W(t_i  w_0, w_0) \le C.
\end{align*}
Since $G$ is compact, by taking a subsequence if necessary, we may assume $\xi _i \to \xi _{\infty} \in G$. Consequently, $d_Z(\xi _i z_0, \xi _{\infty}z_0) \to 0$. However, this implies that $d_Z(t_i  z_0, \xi _{\infty} z_0)$ is uniformly bounded for all $i$, which contradicts Lemma \ref{L403}, 

Thus, each $\R$-orbit in $W$ must also be noncompact.
 \end{proof}

At this stage, the Ricci shrinker equation \eqref{E100} becomes relevant. We aim to derive an equation for the metric $d_W$ on the regular part of $W$, which will ultimately lead to a contradiction. Note that $W$ is the orbit space of the $G$-action on $Z$, and hence the regular part of $W$ corresponds to the principal orbits. For any $z \in Z$, the isotropy group of $G$ at $z$ is contained in the isotropy group of $G$ at $\bar z$, the image of $z$ in $Y$. In particular, this implies that the principal isotropy group of $G$ on $Z$ is finite. The regular part of $W$ is denoted by $\mathcal R_W$, where the metric is given by $g_W$.

Before proceeding with the calculations, we recall the general definition of O’Neill tensors, introduced in \cite{One66}.

\begin{defn}
Let $(M, g_M)$ be a Riemannian manifold with a free, proper, and isometric action by an abelian Lie group $\mathbb T$. The quotient space is $(N, g_N)=(M, g_M)/\mathbb T$, and the projection $(M, g_M) \to (N, g_N)$ is a Riemannian submerison.

For any vector field $E$ on $M$, we can decompose it as the sum of its vertical part $E^v$ (tangent to the $\T$-orbits) and its horizontal part $E^h$ (orthogonal to the $\T$-orbits). The O'Neill tensors $A$ and $T$ are then defined as follows:
\begin{align*}
T(E,F):=(\na_{E^v} F^v)^h+(\na_{E^v} F^h)^v \quad \text{and} \quad A(E,F):=(\na_{E^h} F^h)^v+(\na_{E^h} F^v)^h,
\end{align*}
where $E$ and $F$ are arbitrary vector fields on $M$, and $\na$ is the Levi-Civita connection of $g_M$.
\end{defn}

We fix a basis $\{V^a\}$ of the Lie algebra $\mathfrak t$ of $\mathbb T$, which induces fundamental vector fields, also denoted by $\{V^a\}$, on $M$. Specifically, for any $x \in M$, the vector $V^a$ at $x$ is defined as 
\begin{align*}
\left. \frac{d}{dt} \right\vert_{t=0}\exp(tV^a)\cdot x.
\end{align*}

Additionally, we choose a local basis $\{X^i=\partial_{x_i}\}$ on $N$, which can also be regarded as horizontal vector fields on $M$. Using these bases, we define:
\begin{align} \label{eq:408}
h_{ab}=g_M(V^a, V^b) ,\quad T_{abi}=g_M(T(V^a, V^b), X^i) \quad \text{and} \quad A_{ija}=g_M(A(X^i, X^j), V^a).
\end{align}
Here, $a,b,c,\dots$ denote vertical components, and $i,j,k,\dots$ denote horizontal components. Since $\mathbb T$ is abelian, it follows that $h_{ab}$ is $\mathbb T$-invariant and therefore depends only on $N$. Moreover, we define:
\begin{align} \label{eq:408a}
\mathbf{H}=\sum_{a,b} h^{ab} T(V^a, V^b) \quad \text{and} \quad \mu:=-\log \sqrt{\det (h_{ab})},
\end{align}
where $(h^{ab})$ is the inverse matrix of $(h_{ab})$. 

Here, $\mathbf{H}$ represents the mean curvature vector of the fibers, and $\na \mu=\mathbf{H}$. If a different basis $\{V^a\}$ is chosen, $\mu$ will differ only by a constant.

Using O'Neill's formula \cite{One66}, we have the following identity (see \cite[Proposition 8.1]{NT18}):

\begin{lem}\label{L405}
With the above definitions and notations, we have the following identity on $(N, g_N)$:
\begin{align*} 
\Rc(g_N)_{ij}+(\na^2_{g_N} \mu)_{ij}=\Rc(g_M)_{ij}+h^{ab}h^{cd}T_{aci} T_{bdj}+2h^{ab} g_M^{kl} A_{ika} A_{jlb} \ge \Rc(g_M)_{ij}.
\end{align*}
\end{lem}

\begin{rem}
In the general case where the Lie group $\mathbb T$ is not necessarily abelian, one can consider the adjoint bundle $(M \times \mathfrak t)/\mathbb T$ over $N$. The sections of this adjoint bundle correspond to $\mathbb T$-invariant vertical vector fields $\{V^a\}$ on $M$. The metric $(h_{ab})$, as defined in \eqref{eq:408}, can be interpreted as a metric on this adjoint bundle. For further details and additional identities involving the Ricci curvatures of the total space, fibers, base manifold, and $h_{ab}$, see \emph{\cite[Proposition 8.1]{NT18}}.
\end{rem}

The following result is straightforward:

\begin{lem}\label{L405a}
Under the above assumptions, let $f_M$ be a smooth $\T$-invariant function on $M$. Define $f_N$ as the corresponding function on $N$. Then, the following identity holds:
\begin{align*} 
(\na^2_{g_N} f_N)_{ij}=(\na^2_{g_M} f_M)_{ij}.
\end{align*}
\end{lem}

\begin{proof}
Since $f_M$ is $\T$-invariant, $\na_{g_M} f_M$ is the horizontal lift of $\na_{g_N} f_N$. Thus, we compute:
\begin{align*} 
(\na^2_{g_M} f_M)_{ij}=&g_M\lc \na_{g_M, X^i} \na_{g_M} f_M, X^j \rc=g_M\lc \overline{\na_{g_N, X^i} \na_{g_N} f_N}+\frac{1}{2}[X^i, \na_{g_M} f_M]^v, X^j \rc \\
=&g_N \lc \na_{g_N, X^i} \na_{g_N} f_N, X^j \rc=(\na^2_{g_N} f_N)_{ij},
\end{align*}
where $\bar{\cdot}$ denotes the horizontal lift. This completes the proof.
\end{proof}

We now proceed with the proof of the main theorem. Fix a regular point $w \in \mathcal R_W \subset W$, and let $z \in Z$ be a lift of $w$. Suppose the isotropy group of $G$ at $z$ is $H$, which is a finite group. By the slice theorem, there exists a slice $S_z$ through $z$ and a neighborhood $B$ of $z$ in the form of $G_0 \times_H S_z$, where $G_0$ is a small contractible open neighborhood of the identity in $G$, treated as a local group. Since $z$ lies in the principal orbit, the finite group $H$ fixes $S_z$, and $S_z$, equipped with the quotient metric of $g_Z$, is isometric to a neighborhood of $w$. By our construction of $Z$, we may assume $G_0$ and $S_z$ are sufficiently small so that $B$ is contained in a piece $\R \times V_{\alpha}$.

We denote the preimage of $B$, with respect to $\R \times T^{k-1} \times V_{\alpha} \to \R \times V_{\alpha}$ by $B'$. Moreover, $\R \times T^{k-1} \times V_{\alpha}$ is equipped with the quotient metric of $\tilde h^{\alpha}$, denoted by $h_T^{\alpha}$, from 
\begin{align*} 
\R^k \times V_{\alpha}=\R \times \R^{k-1} \times V_{\alpha} \to \R \times T^{k-1} \times V_{\alpha},
\end{align*}
where $\R \times T^{k-1} \times V_{\alpha}$ comes from the trivialization \eqref{eq:402xxa} and the decomposition \eqref{eq:402xa}.

By our construction, and by taking a smaller $G_0$ if necessary, we may assume that $G_0$ can be lifted to a local action on $\R \times T^{k-1} \times V_{\alpha}$. Furthermore, under the quotient by $T^{k-1}$, this local group action agrees with the action on $\R \times V_{\alpha}$.

Recall from Lemma \ref{L401a} that $\tilde h^{\alpha}$ on $\R^k \times V_{\alpha}$ is the limit of $\tilde g_i^{\alpha}$, which is the lift of $\tilde g_i$. By Theorem \ref{thm:dia} (b) and the Ricci shrinker equation \eqref{E100}, the quotient metric of $h_T^{\alpha}$ on $B''=B'/G_0$, denoted by $h^{\alpha}$, coupled with the quotient function $f^{\alpha}$, satisfies the Ricci shrinker equation.

Since $B'=S_z \times_H G_0 \times T^{k-1}$ and the $T^{k-1}$-action commutes with the $G_0$-action, it follows that $B''= S_z \times (T^{k-1}/H)$, where $T^{k-1}/H$ is also a $(k-1)$-dimensional torus since $H$ is finite. By construction, $(B'', h^{\alpha},  f^{\alpha} )/ (T^{k-1}/H)$ is locally isometric to $(W, g_W, f_W)$ around $w$. 

Combining the Ricci shrinker equation:
\begin{align} \label{eq:410a}
\Rc(h^{\alpha})+\na_{h^{\alpha}}^2(f^{\alpha}) = \frac{1}{2} h^{\alpha},
\end{align}
with Lemma \ref{L405} and Lemma \ref{L405a}, we obtain the following inequality on a small neighborhood of $w$:
\begin{align} \label{eq:410}
\Rc(g_W)+\na_{g_W}^2(f_W+\mu) \ge \frac{1}{2} g_W,
\end{align}
where $\mu$ is the density function in \eqref{eq:408a} with respect to the torus bundle $(B'', h^{\alpha})$. 

We now demonstrate that the density function $\mu$ is globally well-defined on $\mathcal R_W$ and that \eqref{eq:410} holds on $\mathcal R_W$ as well. The density function $\mu(w)$ on $\R \times V_{\alpha}$ is defined so that $|H|e^{-\mu(w)}$, where $|H|$ is the order of the principal isotropy group, represents the volume density of the $T^{k-1}$-fiber in $\R \times T^{k-1} \times V_{\alpha}$ over $w \in \R \times V_{\alpha}$, with respect to the metric $h^{\alpha}_T$. Considering the gluing map described in \eqref{eq:402b}, it is straightforward to verify that $\mu$ is globally well-defined on $\mathcal R_W$. Furthermore, by construction, $\mu$ is invariant under the $\R$-action on $W$.

Define $\bar f_W:=f_W+\mu$. Then, $\bar f_W$ is a smooth, $\R$-invariant function on $\mathcal R_W$. From \eqref{eq:410}, we have:
\begin{align} \label{eq:411}
\Rc(g_W)+\na_{g_W}^2(\bar f_W) \ge \frac{1}{2} g_W.
\end{align}

Next, fix a point $w_0 \in \mathcal R_W$. Without loss of generality, we assume $\overline{B_{g_W}(w_0, 2)} \subset \mathcal R_W$. Take a sequence $w_l$ on the same $\R$-orbit as $w_0$ such that $d_W(w_0, w_l) \to +\infty$ as $l \to \infty$. By Lemma \ref{L404}, such a sequence exists. For each $w_l$, consider a minimizing geodesic $\gamma(t)$, defined for $t \in [0,d_l]$, where $d_l=d_W(w_0, w_l)$, such that $\gamma(0)=w_0$ and $\gamma(d_l)=w_l$. Since $\mathcal R_W$ is geodesically convex by Kleiner's isotropy lemma \cite{KL90}, we conclude that $\gamma(t) \subset \mathcal R_W$ for all $t \in [0, d_l]$.

We construct a function $\eta(t)$ defined as follows:
\begin{align*} 
\eta(t)=
\begin{cases} 
t & \text{ for }t \le 1, \\
1 & \text{ for }1 \le t\le d_l-1, \\
d_l-t & \text{ for } d_l-1 \le t\le d_l.
\end{cases}
\end{align*}

Using the second variation formula for the minimizing geodesic $\gamma(t)$, we have
\begin{align} 
\int_0^{d_l} \eta^2 \Rc(g_W)(\gamma'(t),\gamma'(t)) \, dt \le (m'-1)\int_0^{d_l} (\eta')^2 \, dt=2(m'-1),
\label{eq:412}
\end{align}
where $m' $ is the dimension of $\mathcal R_W$. From \eqref{eq:411},
\begin{align} \label{eq:412a}
\Rc(g_W)(\gamma'(t),\gamma'(t)) +\lc \bar f_W(\gamma(t)) \rc''\ge \frac{1}{2}.
\end{align}
Combining \eqref{eq:412a} with \eqref{eq:412}, we obtain:
\begin{align} 
\frac{d_l}{2}-\frac{2}{3}-2 (m'-1) &\le \int_0^{d_l} \eta^2 \lc \bar f_W(\gamma(t)) \rc'' \, dt \notag \\
& = -2\int_0^1 \eta \lc \bar f_W(\gamma(t)) \rc' \, dt+2\int_{d_l-1}^{d_l} \eta \lc \bar f_W(\gamma(t)) \rc' \, dt \notag \\
& \le 2\sup_{t\in[0,1]} |\na \bar f_W |(\gamma(t))+2\sup_{t\in[d_l,d_l-1]} |\na \bar f_W |(\gamma(t)),
\label{eq:413}
\end{align}
where integration by parts is used to derive the equality. 

Since the $\R$-action is isometric, we conclude that the right-hand side in \eqref{eq:413} is bounded by $4\sup_{w \in B_{g_W}(w, 1)} |\na \bar f_W|$, which is a finite constant. However, taking $l \to \infty$ leads to a contradiction in \eqref{eq:413}, as the left-hand side grows without bound while the right-hand side remains finite.

Thus, we have completed the proof of Theorem \ref{thm:001}.

\section{Non-collapsing for smooth metric measure spaces}

In this section, we extend Theorem \ref{thm:001} to a broader class of smooth metric measure spaces, defined as follows.

\begin{defn}
Given a constant $\kappa>0$, let $\mathcal M(n, A, \kappa)$ denote the class of pointed smooth metric measure spaces $(M^n, g, f, p)$ satisfying the following conditions:
\begin{enumerate}[label=\textnormal{(\alph{*})}]
\item $(M^n, g)$ is a complete $n$-dimensional smooth Riemannian manifold with $|\Rm| \le A$.

\item $f$ is a $C^2$-function on $M$ such that $p$ is minimum point, $f(p)=0$, and $|\na^2 f| \le A$.

\item The Bakry-\'Emery condition holds: $\Rc+\na^2 f \ge \kappa g$.
\end{enumerate}
We also define the subclass $\mathcal N(n, A, \kappa) \subset \mathcal M(n, A, \kappa)$ consisting of pointed smooth metric measure spaces where the underlying manifold is simply connected and has a finite second homotopy group.
\end{defn}

It follows directly from our definition that any Ricci shrinker $(M^n, g, f,p) \in \mathcal E(n, A)$ belongs to $\mathcal M(n, C(n)A+1/2, 1/2)$ if $f$ is shifted by a constant such that $f(p)=0$. Furthermore, for any $(M^n, g, f,p) \in \mathcal M(n, A, \kappa)$, we can define the measure $dV_f:=e^{-f}dV_g$. In this setting, $(M^n, g, dV_f)$ is a smooth $\mathrm{RCD}(\kappa, \infty)$ space, as defined in \cite{Gi18}.

\begin{lem}\label{L501}
There exists a constant $C=C(n, A)>0$ such that for any $(M^n, g, f,p) \in \mathcal M(n, A, \kappa)$ and $x \in M$,
\begin{align*}
f(x) \le A d^2_g(p, x) \quad \text{and} \quad \kappa d_g(p, x)-C \le |\na f(x)| \le A d_g(p, x).
\end{align*}
\end{lem}

\begin{proof}
The bound on $f(x)$ and the upper bound on $|\na f(x)|$ follow directly from $|\na^2 f| \le A$ and the assumption that $f(p)=\na f(p)=0$. To establish the lower bound for $|\na f(x)|$, consider a minimizing geodesic $\gamma(t)$ for $t \in [0, d_x]$ with $\gamma(0)=p$ and $\gamma(d_x)=x$, where $d_x=d_g(p,x)$. 

Using the second variation formula and reasoning similarly to \eqref{eq:413}, we obtain:
\begin{align*} 
\kappa d_x-\frac{2}{3}-2(n-1) \le 2\sup_{t\in[0,1]} |\na f |(\gamma(t))+2\sup_{t\in[d_x,d_x-1]} |\na f |(\gamma(t)) \le 4A+|\na f(x)|,
\end{align*}
where the last inequality uses the bound $|\na^2 f| \le A$. This completes the proof.
\end{proof}

\begin{lem} \label{lem:501a}
There exists a constant $C=C(n, A)>0$ such that for any $(M^n, g, f,p) \in \mathcal M(n, A, \kappa)$ and $x \in M$,
\begin{align*} 
f(x) \ge \frac{\kappa}{2} d^2_g(p, x)-C d_g(p, x)-C.
\end{align*}
\end{lem}

\begin{proof}
Set $r=2C/\kappa$, where $C$ is the constant in Lemma \ref{L501}. Then, for $d_g(p,x) \ge r$, we have
\begin{align} \label{eq:500x}
|\na f(x)| \ge \kappa d_g(p,x)-C \ge C.
\end{align}
For any $y \notin \bar B(p, r)$, we claim that there exists $x \in \partial B(p,r)$ such that $\varphi^b(x)=y$ for some $b>0$, where $\{\varphi^t\}$ is a family of diffeomorphisms generated by $\na f/|\na f|^2$ with $\varphi^0=\mathrm{id}$.

If this is not true, it implies that $\varphi^t(y) \notin \bar B(p, r)$ for any $t \le 0$. Them, for all $t \le 0$,
\begin{align*} 
f(\varphi^t(y))=f(y)+t.
\end{align*}
This implies there exists $t_0<0$ such that $f(\varphi^{t_0}(y))<0$, which contradicts the assumption that $0$ is the minimum value of $f$.

By definition, $f(y)=f(x)+b$. Let $d(t)=d_g(p, \varphi^t(x))$. For any $t \in [0, b]$, we have
\begin{align*} 
d'(t) \le \frac{1}{|\na f|} \le \frac{1}{\kappa d(t)-C},
\end{align*}
where the last term is positive by \eqref{eq:500x}. Integrating this inequality yields:
\begin{align*} 
\frac{\kappa}{2} (d^2(b)-d^2(0))-C (d(b)-d(0)) \le b =f(y)-f(x).
\end{align*}
Since $f(x) \ge 0$ and $d(0) \le r$, the conclusion follows.
\end{proof}

Given $(M^n, g_0, f_0,p_0) \in \mathcal M(n, A, \kappa)$, we consider the Ricci flow $(M^n, g(t))_{t \in [0, T]}$ with $g(0)=g_0$. The existence of the Ricci flow flow is guaranteed by \cite{Shi89B} since $|\Rm(g_0)| \le A$ by our assumption. Furthermore, from \cite[Lemma 6.1]{CLN06}, there exists a constant $\ep=\ep(n, A)=c(n)/A \le 1$ such that
\begin{align*}
|\Rm(g(t))| \le 2A
\end{align*}
for $t \in [0, \ep]$. To handle the potentially noncompact manifold $M$, we need the following exhaustion function, the existence of which is established in \cite[Lemma 4.5]{Shi97}.

\begin{lem}\label{L502}
There exists a smooth function $\psi$ on $M$ and a constant $C=C(n, A)>0$ such that on $M \times [0, \ep]$,
\begin{align*}
C^{-1} (1+d_t(p_0,x)) \le \psi(x)  &\le C (1+d_t(p_0,x)),\\
|\na \psi|+|\na^2 \psi| &\le C.
\end{align*}
\end{lem}

We also need the following heat kernel estimate from \cite[Corollary 5.6]{CTY11}:

\begin{lem}\label{L503}
Let $H(x,t,y,s)$ be the heat kernel of Ricci flow $(M ,g(t))_{t \in [0, \ep]}$. Then, 
\begin{align*}
H(x, t, y,0) \le \frac{C}{|B_{g_0}(x, \sqrt t)|} e^{-\frac{d_0^2(x, y)}{C t}},
\end{align*}
for $C=C(n, A)>0$, where $|\cdot |$ denotes the volume with respect to $g_0$.
\end{lem}

\begin{prop} \label{prop:smoothx}
For the Ricci flow $(M,g(t))_{t \in [0,\ep]}$, there exists a solution of the heat equation
\begin{align} \label{eq:502s}
\square f=0 ,\quad f(0)=f_0,
\end{align}
on $M \times [0,\ep]$, where $\square:=\partial_t-\Delta_t$, such that the following properties hold:
\begin{enumerate}[label=\textnormal{(\alph{*})}]
\item There exists $C=C(n, A)>0$ such that
\begin{align} \label{eq:502}
0 \le f(x, t) \le C \lc f_0(x)+d_0(p_0, x)+1 \rc
\end{align}
for all $(x, t) \in M \times [0, \ep]$. In particular, 
\begin{align} \label{eq:502a}
0 \le f(x, t) \le C \lc d_0^2(p_0, x)+1 \rc.
\end{align}

\item For any constant $\alpha>0$, the following integral is finite:
\begin{align} \label{eq:502b}
\int_0^\ep \int_M \lc |\na f(x)|^2+|\na^2 f (x)|^2 \rc e^{-\alpha d_0^2(p_0,x)} \, dV_t(x) dt <\infty.
\end{align}
\end{enumerate}
\end{prop}

\begin{proof}
We define $f$ as
\begin{align*}
f(x, t)=\int_M H(x, t, y, 0) f_0(y) \,dV_0(y).
\end{align*}
(a): The function $f$ is well-defined by Lemma \ref{L501} and Lemma \ref{L503}. To verify this, we estimate $f(x, t)$ as follows. Set $B_k=B_{g_0}(x, k \sqrt t)$ and write
\begin{align*}
0 \le f(x, t) \le \frac{C}{|B_1|} \sum_{k=1}^{\infty} \int_{B_{k}\setminus B_{k-1}} e^{-\frac{(k-1)^2}{C}} f_0(y) \,dV_0(y).
\end{align*}
Using Lemma \ref{L501}, $|\na f_0| \le A(d_0(p_0, x)+k\sqrt t)$ on $B_k$. Hence, for $y \in B_k$,
\begin{align*}
f_0(y) \le f_0(x)+A(d_0(p_0, x)+k\sqrt t) k \sqrt t \le f_0(x)+A k(d_0(p_0, x)+k).
\end{align*}
Substituting this into the integral, we obtain
\begin{align*}
f(x, t) \le & \frac{C}{|B_1|} \sum_{k=1}^{\infty} |B_k| \lc f_0(x)+A k(d_0(p_0, x)+k) \rc e^{-\frac{(k-1)^2}{C}} \\
\le & C \sum_{k=1}^{\infty} e^{Ck} \lc f_0(x)+A k(d_0(p_0, x)+k) \rc e^{-\frac{(k-1)^2}{C}},
\end{align*}
where we have used the volume comparison estimate $|B_k| \le C e^{Ck} |B_1|$. Thus, $f(x, t)$ satisfies the bound
\begin{align*}
f(x, t) \le C \lc f_0(x)+d_0(p_0, x)+1 \rc.
\end{align*}
In addition, \eqref{eq:502a} follows from Lemma \ref{L501}.

(b): We fix a cutoff function $\eta$ on $\R$ such that $\eta=1$ on $(-\infty, 1]$ and $\eta=0$ on $[2, \infty)$. Using this, we define
\begin{align*}
\phi^r(x):=\eta\lc \frac{\psi(x)}{r} \rc
\end{align*}
for $r \ge 1$, where $\psi(x)$ is the exhaustion function from Lemma \ref{L502}. From the properties of $\psi(x)$, it follows that
\begin{align}\label{eq:504a}
|\na \phi^r(x)|+|\na^2 \phi^r(x)| \le C r^{-1} \le C.
\end{align}

Next, we compute:
\begin{align*}
\partial_t \int f^2 \phi \,dV_t =\int (\square f^2) \phi^r-(\square^* \phi^r) f^2 \,dV_t
\end{align*}
where $\square^*:=-\partial_t-\Delta_t+R$. Note that
\begin{align*}
\square f^2=-2|\na f|^2,
\end{align*}
and using \eqref{eq:504a}, we have
\begin{align*}
|\square^* \phi^r| \le C.
\end{align*}

Integrating over time, we obtain for $r \ge 1$:
\begin{align*}
\int_0^{\ep} \int_M 2|\na f|^2 \phi^r \,dV_t dt \le \left. \int f^2 \phi^r \,dV_t \right \vert_{t=0}+C \int_0^{\ep} \int_{\mathrm{supp}(\phi^r)} f^2 \,dV_t dt
\end{align*}

From \eqref{eq:502a} and Lemma \ref{L502}, the right-hand side is bounded by
\begin{align*}
C(1+r^2)^2 |B_{g_0}(p_0, C r)|,
\end{align*}
where we have used the fact that $dV_t$ is comparable to $dV_0$. Thus, we obtain:
\begin{align}\label{eq:507}
\int_0^{\ep} \int_{B_{g_0}(p_0, r)} |\na f|^2 \,dV_t dt \le C(1+r^2)^2 |B_{g_0}(p_0, C r)|.
\end{align}
Using the volume comparison, this implies:
\begin{align}\label{eq:508}
\int_0^{\ep} \int_M |\na f (x)|^2 e^{-\alpha d_0^2(p_0,x)}\,dV_t(x) dt <\infty.
\end{align}

For the second term in \eqref{eq:502b}, note that $\square |\na f|^2=-2|\na^2 f|^2$. Following a similar process as above, we have:
\begin{align*}
\int_0^{\ep} \int_M 2|\na^2 f|^2 \phi^r \,dV_t dt \le \left. \int |\na f|^2 \phi^r \,dV_t \right \vert_{t=0}+C \int_0^{\ep} \int_{\mathrm{supp}(\phi^r)} |\na f|^2 \,dV_t dt.
\end{align*}
Using Lemma \ref{L501} and \eqref{eq:507}, we conclude:
\begin{align*}
\int_0^{\ep} \int_{B_{g_0}(p_0, r)} |\na^2 f|^2 \,dV_t dt \le C (1+r^2+r^4) |B_{g_0}(p_0, C r)|,
\end{align*}
and thus:
\begin{align}\label{eq:511}
\int_0^{\ep} \int_M |\na^2 f (x)|^2 e^{-\alpha d_0^2(p_0,x)}\,dV_t(x) dt <\infty.
\end{align}
Combining \eqref{eq:508} and \eqref{eq:511}, we have completed the proof of \eqref{eq:502b}.
\end{proof}

From now on, $f$ denotes the solution of \eqref{eq:502s} obtained in Proposition \ref{prop:smoothx}. Next, we prove:

\begin{prop} \label{prop:smooth1}
There exists a constant $C=C(n, A)>0$ such that on $M \times [0, \ep]$,
\begin{align}
|\na^2 f| \le C \quad \text{and} \quad |\na f(x, t)| \le & C \lc d_0(p_0, x)+1 \rc. \label{eq:513}
\end{align}
\end{prop}

\begin{proof}
We begin by considering the evolution equation for $\na^2 f$:
\begin{align*}
(\partial_t-\Delta_L) (\na^2 f)=\na^2 (\square f)=0,
\end{align*}
where $\Delta_L$ denotes the Lichnerowicz Laplacian. Using the curvature bounds, we compute:
\begin{align*}
\square |\na^2 f|^2 \le -2|\na^3 f|^2+C|\na^2 f|^2 \le C|\na^2 f|^2.
\end{align*}
By the maximum principle (see \cite[Theorem 7.42]{CLN06}), we conclude
\begin{align*}
|\na^2 f| \le C. 
\end{align*}
Note that the assumptions required for applying the maximum principle are guaranteed by Proposition \ref{prop:smoothx} (b).

Next, since $\square |\na f|^2=-2|\na^2 f|^2 \le 0$, we have
\begin{align*}
|\na f(x, t)|^2 \le \int H(x, t, y,0) |\na f_0 (y)|^2 \,dV_0(y).
\end{align*}

With the help of Lemma \ref{L501} and Lemma \ref{L502}, and following a similar argument as in the proof of \eqref{eq:502}, we conclude that
\begin{align*}
|\na f(x, t)|^2 \le C(d_0(p_0, x)+1)^2.
\end{align*}
\end{proof}

\begin{prop} \label{prop:smooth2}
There exists a small constant $\ep_1=\ep_1(n, A ,\kappa) \le \min\{\ep, \frac{1}{4\kappa}\}$ such that on $M \times [0, \ep_1]$,
\begin{align*}
\Rc(g(t))+ \frac{1}{1-2\kappa t} \na^2_{g(t)} f(t) \ge \frac{\kappa}{2} g(t).
\end{align*}
\end{prop}

\begin{proof}
A direct calculation shows that
\begin{align} \label{eq:514} 
(\partial_t-\Delta_L) \lc (1-2\kappa t) \Rc+\na^2 f-\kappa g \rc =0,
\end{align}
where $\Delta_L$ is the Lichnerowicz Laplacian. Define the $(0,2)$-tensor
\begin{align*}
h=(1-2\kappa t) \Rc+\na^2 f-\kappa g,
\end{align*}
and let $\lambda$ denote the minimal eigenvalue of $h$. By assumption, $\lambda \ge 0$ at $t =0$. 

Fix $(x_0, t_0) \in M \times [0, \ep]$, and let $v$ be a unit eigenvector corresponding to $\lambda$ at $(x_0,t_0)$. Extend $v$ to be constant in time for $t \in [t_0-\delta, t_0+\delta]$, and then extend it spatially along each geodesics originating at $(x_0, t)$, using parallel transport with respect to $g(t)$. 

Define 
\begin{align*}
\phi(x, t)=h(v(x,t), v(x, t)),
\end{align*}
which is defined on a small spacetime neighborhood of $(x_0, t_0)$. By substituting into the evolution equation \eqref{eq:514}, we find that at $(x_0, t_0)$,
\begin{align*}
\partial_t \phi= \Delta \phi+2\Rm(h)(v, v)-2g( \Rc(v), h(v)) \ge \Delta \phi-C(n, A ,\kappa),
\end{align*}
where the last inequality follows from the curvature bounds and the estimate \eqref{eq:513}. 

Thus, in the barrier sense, we obtain the inequality
\begin{align*} 
\partial_t \lambda \ge \Delta \lambda-C(n, A ,\kappa).
\end{align*}
Applying the maximum principle, we conclude that
\begin{align*}
(1-2\kappa t) \Rc+\na^2 f\ge (\kappa-Ct) g.
\end{align*}

Therefore, the conclusion follows if $t$ is small.
\end{proof}

We also need the following Shi-type estimates for higher-order terms (see \cite{Shi89A}).

\begin{prop} \label{prop:smooth3}
For each $k \ge 0$, there exists a constant $C_k=C_k(n, A)>0$ such that for $t \in (0, \ep]$,
\begin{align*} 
|\na^k \Rm|^2 +|\na^{k+2} f|^2 \le C_k t^{-k-1}.
\end{align*}
\end{prop}

We now prove the following result.

\begin{thm} \label{thm:smooth}
For any $(M^n, g_0, f_0,p_0) \in \mathcal M(n, A, \kappa)$, there exists $(M^n, g, f, p)$ satisfying the following properties:
\begin{enumerate}[label=\textnormal{(\alph{*})}]
\item There exists $A'=A'(n, A, \kappa)>0$ such that
\begin{align*}
(M^n, g, f, p) \in \mathcal M(n, A', \kappa/2).
\end{align*}

\item For each $k \ge 0$, there exists $C_k=C_k(n, A, \kappa)>0$ such that
\begin{align*}
|\na_{g}^k \Rm(g)|^2 +|\na_{g}^{k+2} f|^2 \le C_k.
\end{align*}

\item There exists $C=C(n, A, \kappa)>1$ such that
\begin{align*}
C^{-1} \mathrm{Vol}_{g_0} \lc B_{g_0}(p_0, 1) \rc \le \mathrm{Vol}_{g} \lc B_{g}(p, 1) \rc \le C \mathrm{Vol}_{g_0} \lc B_{g_0}(p_0, 1) \rc.
\end{align*}
\end{enumerate}
\end{thm}

\begin{proof}
For the constant $\ep_1$ in Proposition \ref{prop:smooth2}, we define
\begin{align*}
g:=g(\cdot,\ep_1) \quad \text{and} \quad f:=\frac{1}{1-2\kappa \ep_1} f(\cdot,\ep_1).
\end{align*}
From Proposition \ref{prop:smooth2}, we have
\begin{align}\label{eq:518}
\Rc(g)+ \na_{g}^2 f \ge \frac{\kappa}{2} g.
\end{align}
Additionally, from Propositions \ref{prop:smooth1} and \ref{prop:smoothx} (a):
\begin{align}\label{eq:519}
|\na_{g}^2 f| \le C(n,A, \kappa) \quad \text{and} \quad |\na_{g} f(p_0)|+f(p_0) \le C(n,A, \kappa).
\end{align}
Using the arguments in Lemma \ref{L501}, we deduce for any $x \in M$,
\begin{align*}
|\na_{g} f(x)| \ge \frac{\kappa}{2} d_{g}(p_0, x)-C(n,A, \kappa).
\end{align*}
As $f \ge 0$, it follows (via a similar approach to Lemma \ref{lem:501a}) that:
\begin{align*} 
f(x) \ge \frac{\kappa}{4} d_{g}^2(p_0, x)-C(n,A, \kappa) d_{g}(p_0, x)-C(n,A, \kappa).
\end{align*}
Consequently, $f$ attains its minimum at some point $p \in M$, and
\begin{align*} 
d_{g}(p, p_0) \le C(n,A, \kappa).
\end{align*}
By adding a constant to $f$, we can normalize $f(p)=0$. 

Assertion (a) follows directly from \eqref{eq:518} and \eqref{eq:519}. Assertion (b) is established using Proposition \ref{prop:smooth3}. Finally, assertion (c) follows from the standard volume comparison argument, using the fact that $d_{g_0}$ and $d_{g}$ (as well as $dV_{g_0}$ and $dV_{g}$) are comparable.
\end{proof}

With these preparations, we are now ready to prove Theorem \ref{thm:002}, restated here for convenience.

\begin{thm}
There exists a constant $C=C(n, A, \kappa)>0$ such that for any $(M^n, g, f,p) \in \mathcal N(n, A, \kappa)$,
\begin{align*}
\mathrm{Vol}_{g} \lc B_{g}(p, 1) \rc \ge C.
\end{align*}
\end{thm}

\begin{proof}
We proceed by contradiction. Assume there exists a sequence $(M_i^n, g_i, f_i,p_i) \in \mathcal N(n, A, \kappa)$ such that
\begin{align*}
\mathrm{Vol}_{g_i} \lc B_{g_i}(p_i, 1) \rc \to 0.
\end{align*}
By Theorem \ref{thm:smooth}, we may assume that for some constants $C_k=C_k(n, A, \kappa)$,
\begin{align*}
|\na_{g_i}^k \Rm(g_i)|^2 +|\na_{g_i}^{k+2} f_i|^2 \le C_k.
\end{align*}

The rest of the proof follows verbatim as the proof of Theorem \ref{thm:001}. The estimates for the potential function and its gradient are derived using Lemma \ref{L501} and Lemma \ref{L502}. 

The main distinction here is the absence of the Ricci shrinker equation \eqref{E100}. Instead, we rely on the weaker inequality:
\begin{align*}
\Rc(g_i)+\na^2_{g_i} f_i \ge \kappa g_i.
\end{align*}
This affects the formulation of \eqref{eq:410a}, which now becomes
\begin{align*}
\Rc( h^{\alpha})+\na_{ h^{\alpha}}^2( f^{\alpha}) \ge \kappa  h^{\alpha}.
\end{align*}
However, the key inequalities \eqref{eq:410} and \eqref{eq:411} remain unaffected, except for replacing the coefficient $1/2$ with $\kappa$. Therefore, we derive the same contradiction as before.
\end{proof}

    \vskip10pt
    
Conghan Dong, Simons Laufer Mathematical Sciences Institute, 17 Gauss Way, Berkeley, CA 94720, United States, Email: conghan.dong@duke.edu.\\
    
Yu Li, Institute of Geometry and Physics, University of Science and Technology of China, No. 96 Jinzhai Road, Hefei, Anhui Province, 230026, China; Hefei National Laboratory, No. 5099 West Wangjiang Road, Hefei, Anhui Province, 230088, China; E-mail: yuli21@ustc.edu.cn. \\

    \end{document}